\documentclass[11pt,reqno]{amsart}
\usepackage[utf8]{inputenc}
\usepackage{amsmath,amssymb,amsthm,mathrsfs, color}
\usepackage{a4wide}
\usepackage{epsfig}
\usepackage{caption}
\usepackage{subcaption} 
\usepackage{graphics}
\usepackage{epsfig}

\usepackage[colorlinks=true]{hyperref}
\hypersetup{%
    ,urlcolor=blue
    ,citecolor=red
    ,linkcolor=blue
    }

\newtheorem{theorem}{Theorem}[section]
\newtheorem{prop}[theorem]{Proposition}
\newtheorem{lemma}[theorem]{Lemma}
\newtheorem{corollary}[theorem]{Corollary}
\newtheorem{definition}[theorem]{Definition}
\newtheorem{remark}[theorem]{Remark}
\newtheorem{example}[theorem]{Example}

\newcommand{\cM}{\mathcal{M}}

\newcommand{\R}{{\mathbb R}}
\newcommand{\N}{{\mathbb N}}
\newcommand{\E}{{\mathbb E}}

\newcommand{\cF}{{\mathcal F}}
\newcommand{\cL}{{\mathcal L}}
\newcommand{\bfx}{{\mathbf x}}
\newcommand{\bfX}{{\mathbf X}}
\newcommand{\scrC}{{\mathscr C}}
\renewcommand{\P}{{\mathbb P}}
\newcommand{\Cpol}[1]{{\mathcal C}^\infty_{\textup{pol}}(#1)}

\def\EKG{\mathscr{E}_K^G}

\def\blue{\textcolor{black}}
\makeatletter
 \@addtoreset{equation}{section}
 \makeatother
 
\usepackage{appendix}
\usepackage{hyperref}

\title[Nonnegativity preserving convolution kernels. Application to SVEs]{Nonnegativity preserving convolution kernels. Application to Stochastic Volterra Equations in closed convex domains and their approximation.}
\date{\today}
\author{Aurélien Alfonsi}
\address{Aurélien Alfonsi, CERMICS, Ecole des Ponts, Marne-la-Vall\'ee, France. MathRisk, Inria, Paris, France.}
\email{aurelien.alfonsi@enpc.fr}

\thanks{This work benefited from the support of the ``chaire Risques financiers'', Fondation du Risque. \\
I thank the anonymous referees for their  stimulating suggestions on the first version of this paper. }


\begin{document}
\maketitle

\begin{abstract}
This work defines and studies \blue{one-dimensional} convolution kernels that preserve nonnegativity. When the past dynamics of a process is integrated with a convolution kernel like in Stochastic Volterra Equations or in the jump intensity of Hawkes processes, this property allows to get the nonnegativity of the integral. We give characterizations of these kernels and show in particular that completely monotone kernels preserve nonnegativity. We then apply these results to analyze the stochastic invariance of a closed convex set by Stochastic Volterra Equations. We also get a comparison result in dimension one. Last, when the kernel is a positive linear combination of decaying exponential functions, we present a second order approximation scheme for the weak error that stays in the closed convex domain under suitable assumptions. We apply these results to the rough Heston model and give numerical illustrations. 
\end{abstract}

\section{Introduction}

Convolution kernels are widely used in applications to model phenomena with memory that fade away along the time. Without being exhaustive at all, we mention here applications to mechanics~\cite{CoNo,MacCamy}, biology~\cite{DLP,RBSc,PNAS}, sociology~\cite{CrSo,BBH} or finance~\cite{Gatheral,ASS,BDHM,EER}. The main families of stochastic processes used in the related models are Hawkes processes and Stochastic Volterra Equations, where the convolution kernel is used to integrate the past dynamics. For Hawkes processes, this integration on the past produces the jump intensity and has thus to be nonnegative.  For Stochastic Volterra Equations, it is also important for some applications to know if they are nonnegative or more generally if they stay in some domain. 

We first motivate the introduction of nonnegativity preserving convolution kernels. Let us consider a kernel function $G:\R_+\to \R_+$ and an increasing sequence of times $(t_k)_{k\ge 1}$ such that $t_1\ge 0$ and $t_k\to \infty$ as $k\to \infty$. For a sequence of real numbers $(x_k)_{k\ge 1}$, we consider the process 
$$ X_t= \sum_{k: t_k\le t}G(t-t_k) x_k.$$
This models a scalar phenomenon with memory: at each time~$t_k$, the process is increased by~$x_k$, and the memory of this is then scaled in the future by the kernel~$G$. The function~$G$ is typically nonincreasing for a fading memory, which is the most common case in applications. We are interested in the following question: under which condition do we have $X_t\ge 0$ for all $t\ge 0$? This is clearly true if all the $x_k$'s are nonnegative. It is also clear that a necessary condition is to have  $X_{t_K}= \sum_{k=1}^K G(t_K-t_k) x_k\ge 0$ for all $K\ge 1$. In this paper, we are interested in characterizing kernels~$G$ for which this condition is sufficient to guarantee that $X_t\ge 0$ for all $t\ge 0$. We say that these kernels preserve nonnegativity (see Definition~\ref{def_PPKernels}). The main results are the following. \blue{First,} completely monotone kernels, that are widely used in applications, preserve nonnegativity (Theorem~\ref{thm:positivite}). \blue{Perhaps surprisingly, this property does not seem to have been pointed in the literature on completely monotone functions, while it is related to old problems in mathematics, see Remarks~\ref{rk_Schur} and~\ref{rk_Bartholomew} below.} Second, we fully characterize nonnegativity preserving kernels in Theorem~\ref{thm_char_pos} and give in Theorem~\ref{thm_carac_np} a more precise characterization for nonincreasing right-continuous kernels by the mean of their resolvent of the first kind. A remarkable property with càdlàg positivity preserving kernels is the following. Even if~$X$ is a process with memory, there is no need to know all the past to determine at time $t_k$ what values of $x_k$ guarantee the nonnegativity of~$X$ in the future, provided that $X_t\ge 0$ for $t\in [0,t_k)$. It is necessary and sufficient to have $G(0)x_k+X_{t_k-}\ge 0$, and therefore only the present information on~$X$ is needed. \blue{Let us note that this kind of property already show through in the works of Abi Jaber et al.~\cite{AJEE,AJLP,AJ_Bernoulli} for Stochastic Volterra Equations with a nonincreasing continuous convolution kernel having a nonnegative and nonincreasing resolvent of the first kind. However, this property appears in this case from a contraposition while the nonnegativity preserving property allows for a direct implication.  }

A potential important application of this result \blue{on nonnegativity} brings on processes of Hawkes type with inhibition. Hawkes processes with inhibition have been recently studied by Costa et al.~\cite{CGMT}, Chen et al.~\cite{CSSBW}, Bonde Raad and L\"ocherbach~\cite{BRL}, Cattiaux et al.~\cite{CCC} and Duval et al.~\cite{DLP}. To fix the ideas, let us consider a marked jump process $(N_t,t\ge 0)$ that is right-continuous, piecewise constant, starts from $N_0\in \R$ at time~$0$ and has the jump intensity 
$$ X_t=x_0 + \int_0^t G(t-s) dN_s, \ t\ge 0,$$
where $x_0>0$ and $G$ is a càdlàg nonincreasing kernel that preserves nonnegativity. In most of the literature of Hawkes processes, the kernel $G$ and the jumps of~$N$ are assumed to be nonnegative so that the process is self-exciting. To include an inhibition behaviour, one has to take a signed kernel as in~\cite{CGMT} or to allow negative jumps, but there is then no guarantee to have a nonnegative intensity. Thus, it is assumed in~\cite{CGMT} that the intensity of~$N$ is given by $(X_t)^+$ or more generally by $\phi(X_t)$, where $\phi: \R\to \R_+$ is a nondecreasing function. The nonnegativity preserving property of~$G$  
allows to get the nonnegativity "naturally", without the use of such a function. Namely, by Proposition~\ref{prop_pos_gen} below, we have $X_t\ge 0$ for all $t\ge 0$ if, and only if 
$$X_{t-}+G(0)dN_t \ge 0,$$ 
at each jump of~$N$. While the process $(N,X)$ is a priori not Markovian, this condition does not involve $(N_s,X_s)_{s\in[0,t)}$ but only $X_{t-}$. We can imagine different choices of models for the conditional law of $dN_t$ given $X_{t-}$ and we leave this for further investigations. Extensions to multidimensional Hawkes processes with excitation and inhibition (self and mutual) could also be developed. Last, let us note that the property mentioned can be useful for stochastic control problems, for example if we consider an impulse control on the jump intensity.
All these investigations are far beyond the scope of the present paper. Except few remarks (namely Remarks~\ref{rk_Hawkes1} and~\ref{rk_Hawkes2}), the paper mostly focuses on some applications to  Stochastic Volterra Equations that we present now. 

Stochastic Volterra Equations (SVE) have been introduced by Berger and Mizel~\cite{BeMi}. These processes extend naturally Stochastic Differential Equations and allow to take into account memory in their dynamics. In this work, we only focus on these equations with a \blue{real valued} convolution kernel~$G:\R_+ \to \R_+$ and consider the process $X$ defined by
\begin{equation}\label{intro_to_SVE} X_t=x_0 + \int_0^t G(t-s) b(X_s)ds +\int_0^t G(t-s) \sigma (X_s) dW_s,
\end{equation}
where $x_0 \in \R^d$, $b:\R^d \to \R^d$ and $\sigma: \R^d \to \mathcal{M}_d(\R)$ are respectively vector and matrix valued functions, and $W$ is a $d$-dimensional Brownian motion. Under suitable conditions on the kernel and the coefficients $b$ and $\sigma$, there is indeed a unique solution to this equation, see~\cite[Theorem 3.A]{BeMi}. Stochastic Volterra Equations of this type have known a growing interest in financial applications to model the volatility process since the seminal work by Gatheral et al. on rough volatility~\cite{GJR}. An important example is the rough Heston model proposed by El~Euch and Rosenbaum~\cite{EER}, where the volatility process follows a one-dimensional SVE~\eqref{intro_to_SVE} with  an affine function $b$, a square-root function $\sigma$ and a fractional kernel~$G$. It is shown that this process, called rough Cox-Ingersoll-Ross process, remains nonnegative. For this and for other potential applications, it is important to determine that the process $X$ stays in some domain. While the stochastic invariance is well studied for Stochastic Differential Equations (see Abi Jaber et al.~\cite{AJBI} and references therein), there exist up to our knowledge very few works in the literature on similar results for SVEs. We mention here the work by Abi Jaber et al.~\cite{AJLP} that obtains the existence of a weak solution to~\eqref{intro_to_SVE} on $\R_+^d$ under suitable assumptions on $b$ and $\sigma$, and by assuming that the kernel~$G$ is nonincreasing, continuous with a nonnegative and nonincreasing resolvent of the first kind. \blue{Abi Jaber~\cite{AJ_Bernoulli} shows also the nonnegativity of SVEs with jumps under the same condition on the kernel.}  In the second part of this work, we give sufficient conditions ensuring that the SVE stay in some domain~$\scrC \subset \R^d$. Basically, when $\scrC$ is a closed convex set and the kernel $G$ preserves nonnegativity, the conditions on $b$ and $\sigma$ are the same as those required for SDEs (see Theorem~\ref{thm_SVE_dom}). We obtain this result by using an approximation of the SVE that consists in splitting the convolution part and the SDE part of the dynamics.  We also get these results for completely monotone kernels such that $G(0+)=+\infty$, which includes fractional kernels. Besides, we obtain in dimension one a comparison result between SVEs (Theorem~\ref{thm_comparison_result}). \blue{These results are obtained under the standard assumption of Lipschitz coefficients. However, the same approximation procedure could also be used to get stochastic invariance results for weak solutions of SVEs. This is left for further research (see Remark~\ref{rk_Wishart} below). }

The last part of this work focuses on the approximation of SVEs. More precisely, we are interested in developing approximation schemes that are well defined on~$\scrC$ for SVEs that stays in this domain. By using splitting technique, we introduce in Theorem~\ref{thm_second_order_SVE} such a scheme that leads to a weak approximation error of order~2, when $G(t)=\sum_{i=1}^n\gamma_i e^{-\rho_i t}$ with $n\ge 1$, $\gamma_1,\dots,\gamma_i>0$ and $0\le \rho_1<\dots \rho_n$. We then apply this technique to propose a weak second order scheme for the multifactor Heston model introduced by Abi~Jaber and El~Euch~\cite{AJEE} as a proxy of the rough Heston model. This was the original motivation of this work. Namely, combining our general approach with the second order schemes for the Cox-Ingersoll-Ross (CIR) and Heston model proposed in Alfonsi~\cite{AA_MCOM}, we propose new approximation schemes for the multifactor version of these models. We observe well on numerical experiments a weak convergence of order two. We also compare our scheme to the Euler schemes developed in Richard et al.~\cite{RTY}  and Alfonsi and Kebaier~\cite{AK} for option pricing in the rough Heston model and shows that it outperforms these schemes in terms of bias and computation time.   

The paper is organized as follows. Section~\ref{Sec_kernels} is devoted to the analysis on nonnegativity preserving kernels. We first give the definition and obtain a first characterization of these kernels that does not assume any regularity. Then, we focus on completely monotone kernels and show that they preserve nonnegativity. Last, we study the relation between the existence of a nonnegative nonincreasing resolvent of the first kind and the property of preserving nonnegativity. Section~\ref{Sec_SVEs} is dedicated to the stochastic invariance. The first part presents the approximation of SVEs and shows that it stays in some convex domain by taking advantage of the nonnegativity preserving property. We estimate the strong error and then state our main results on stochastic invariance that are obtained by sending the time step to zero. Section~\ref{Sec_approx} develops approximation schemes for SVEs that have a weak error of order two and that stay under suitable assumptions in a closed convex domain. It first recalls general results on the weak approximation of SDEs and then presents a second order scheme for SVEs for the weak error. Last, it develops second order schemes for the multifactor CIR and Heston models and presents some numerical results in this framework.

\section{Nonnegativity preserving convolution kernels}\label{Sec_kernels}
\subsection{Definition and characterization}

\begin{definition}\label{def_PPKernels}
  Let $G:\R_+ \to \R_+$ be a function (called kernel) such that $G(0)>0$.  We say that the kernel $G$ preserves nonnegativity if, for any $K\in \N^*$ and any $x_1,\dots,x_K \in \R$ and $0\le t_1< \dots<t_K$ such that
  $$\forall k \in \{1, \dots, K\},\ \sum_{k'=1}^k x_{k'}G(t_k-t_{k'})\ge 0, $$
  we have
  \begin{equation}\label{posit}
    \forall t \ge 0, \sum_{k :  t_k \le t} x_k G(t-t_k) \ge 0 \quad (\text{with the convention } \sum_{\emptyset}=0).
  \end{equation} 
\end{definition}
Note that in this definition, $x_1$ is nonnegative, but $x_k$ may be negative for some $k\ge 2$. We start by presenting the most simple example of nonnegativity preserving kernels. 
\begin{example}\label{example_expo}
Exponential functions $G(t)=\lambda e^{-\rho t}$ with $\lambda>0$ and $\rho \in \R$ clearly preserve nonnegativity. In this case, we have $\sum_{k :  t_k \le t} x_k G(t-t_k)=e^{-\rho(t-t_k)} \sum_{k'=1}^k x_{k'}G(t_k-t_{k'})\ge 0$ for $t_k\le t< t_{k+1}$ (convention $t_{K+1}=+\infty$). 
\end{example}
\noindent Let us stress that we assume in Definition~\ref{def_PPKernels} that $G(0)$ is finite. It would be meaningless to extend Definition~\ref{def_PPKernels} to kernels $G:\R_+^*\to \R_+$ such that $G(0+)=\lim_{t\to 0, t>0}G(t)$ exists and $G(0+)=+\infty$. In fact, if we have  
$\sum_{k :  t_k \le t} x_k G(t-t_k) \ge 0$ for all $t\ge 0$ such that $t\not \in \{t_k,1\le  k\le K\}$, then we get $x_k\ge 0$ for all $k$ by taking the limit in $t_k+$. Conversely, we clearly have $\sum_{k :  t_k \le t} x_k G(t-t_k) \ge 0$ when all the $x_k$'s are nonnegative. For those kernels, we will rather try when needed to approximate them by a sequence of positivity preserving kernels, as later on in Theorems~\ref{thm_comparison_result} and~\ref{thm_SVE_dom}. 

\begin{remark} It would be natural to extend the notion of preserving nonnegativity to double kernels $\Gamma:\{(s,t)\in \R_+: s\le t \} \to \R_+$ as follows: 
  for any $K\in \N^*$ and any $x_1,\dots,x_K \in \R$ and $0\le t_1< \dots<t_K$ such that
  $$\forall k \in \{1, \dots, K\},\ \sum_{k'=1}^k x_{k'}\Gamma(t_{k'},t_k)\ge 0, $$
  we have $\forall t \ge 0, \sum_{k :  t_k \le t} x_k \Gamma(t_k,t) \ge 0$. In this paper, we only focus on convolution kernels \blue{ since they are widely used,  and one of our goal is to make the connection with the existing literature on stochastic invariance for SVEs. A priori, there is no major difficulty to adapt the results of Section~\ref{Sec_SVEs} to nonnegativity preserving double kernels. A more challenging direction is to find relevant examples of nonnegativity preserving double kernels that are not of convolution type.  }
\end{remark} 

We first give a first characterisation of nonnegativity preserving kernels. For a kernel~$G:\R_+ \to \R_+$, we denote
\begin{align}
  \EKG:=\Bigg\{ (x_1,\dots,x_K,t_1,\dots,t_K) \in \R^{2K}& : 0\le t_1< \dots<t_K  \text{ and } \label{def_EKG}\\ 
 &\forall k \in \{1, \dots, K\},\ \sum_{k'=1}^k x_{k'}G(t_k-t_{k'})\ge 0 \Bigg\}.\notag
\end{align}
\begin{definition} For a kernel $G:\R_+ \to \R_+$ such that $G(0)> 0$,  we define by induction, for $l\in \N^*$, the functions $G_l:(\R_+^*)^{l}\to \R$ by   $G_1(a_1)=G(a_1)/G(0)$ for $a_1> 0$ and
  \begin{equation}\label{def_Gl}
    G_l(a_1,\dots,a_l)=G_{l-1}(a_1,\dots,a_{l-2},a_{l-1}+a_l)-G_1(a_l)G_{l-1}(a_1,\dots,a_{l-1}),
  \end{equation}
  for $l\ge 2$, $a_1,\dots,a_l \in \R^*_+$.
\end{definition}

\begin{lemma}\label{lem_char_pos} Let~$G:\R_+ \to \R_+$ such that $G(0)> 0$. Let $K\in \N^*$ and  $0\le t_1 <\dots <t_K < t_{K+1}=t$ . Then,
  \begin{align}& \left( (x_1,\dots,x_K,t_1,\dots,t_K)\in \EKG \implies \sum_{k=1}^K x_kG(t-t_k)\ge 0 \right) \label{implic} \\ \iff &\forall j\in \{1,\dots,K\}, G_j(t-t_K,t_K-t_{K-1},\dots,t_{K-j+2}-t_{K-j+1})\ge 0. \notag
  \end{align}
  Let $0\le t_1 <\dots <t_K$ and define $x_1=1$ and  $x_k=\frac{-1}{G(0)} \sum_{k'=1}^{k-1} x_{k'}G(t_k-t_{k'})$ for $k\ge 2$, so that $\sum_{k'=1}^k x_{k'}G(t_k-t_{k'})=0$. Then, for $k\in \{1,\dots,K-1\}$, we have
  $$x_{k+1}=-G_k(t_{k+1}-t_{k},\dots,t_2-t_1).$$
\end{lemma}
Lemma~\ref{lem_char_pos} allows to characterize nonnegativity preserving kernels, that amounts to have the implication~\eqref{implic} for all $K\in \N$ and $0\le t_1<\dots< t_K< t$. This gives the following theorem. 
\begin{theorem}\label{thm_char_pos}
  Let~$G:\R_+ \to \R_+$ such that $G(0)> 0$. The kernel~$G$ preserves nonnegativity if, and only if all the functions $G_l:(\R_+^*)^l\to \R$, $l\in \N^*$, defined by~\eqref{def_Gl} are nonnegative.
\end{theorem}  
The second part of Lemma~\ref{lem_char_pos} considers, roughly speaking, the "worst case" strategies for the nonnegativity that start from~$1$ at time $t_1$ and that are brought back to zero at each time $t_k$, for $k\ge 2$. Thanks to Theorem~\ref{thm_char_pos}, we see that the kernel~$G$ preserves nonnegativity if, and only if those strategies remains nonnegative for any $K\in \N^*$ and any times $0\le t_1<\dots<t_K$. We also get the following interesting corollary.

\begin{corollary}\label{cor_expoinvariance}
  Let~$G:\R_+ \to \R_+$ such that $G(0)> 0$. Then, $G:\R_+ \to \R_+$ preserves nonnegativity if, and only if $t\mapsto G(t)e^{-\rho t}$ preserves nonnegativity for any $\rho \in \R$. 
\end{corollary}
\begin{proof}
  Let us define $G^\rho(t)=G(t)e^{-\rho t}$. We consider the associated functions $G_l^\rho:(\R_+^*)^l\to \R$ defined inductively by $G_1^\rho(a_1)=G^\rho(a_1)/G^\rho(0)$ and $G^\rho_l(a_1,\dots,a_l)=G^\rho_{l-1}(a_1,\dots,a_{l-2},a_{l-1}+a_l)-G^\rho_1(a_1)G^\rho_{l-1}(a_1,\dots,a_{l-2},a_{l-1})$. We have $G^\rho_l(a_1,\dots,a_l)=e^{-\rho \sum_{j=1}^l a_j}G_l(a_1,\dots,a_l)$: this is true for $l=1$ and then obvious by induction. Therefore $G_l\ge0 \iff G^\rho_l\ge 0$, and we conclude by Theorem~\ref{thm_char_pos}.
\end{proof}

\begin{proof}[Proof of Lemma~\ref{lem_char_pos}.]
We prove the first part of the statement and consider $0\le t_1<\dots<t_K$.  We are interested in the following minimization problem,  for $t> t_K$:
\begin{equation}  \label{min_pb} \inf_{(x_1,\dots,x_K): (x_1,\dots,x_K,t_1,\dots,t_K)\in \EKG}\ \sum_{k=1}^Kx_kG(t-t_k).
\end{equation}
The implication~\eqref{implic}, is equivalent to the nonnegativity of this infimum, which we want to characterize. We note that the set of constraints is linear and triangular \blue{with nonzero diagonal elements since $G(0)>0$. Thus,}  we can write the linear function to optimize as a linear combination of the constraints:
\begin{equation}\label{lin_to_opt}
  \sum_{k=1}^Kx_kG(t-t_k)= \sum_{k=1}^K \beta_k \left( \sum_{k'=1}^k x_{k'}G(t_k-t_{k'})\right), 
\end{equation}
with
\begin{equation}\label{def_betas}
  \forall k' \in \{1,\dots, K \}, \sum_{k=k'}^K \beta_k G(t_k-t_{k'})  =G(t-t_{k'}).
\end{equation}
We rewrite~\eqref{def_betas} explicitly from $k'=1$ to $k'=K$ as
\begin{equation}
  \begin{cases}
    \sum_{k=1}^K \beta_k G_1(t_k-t_1)  =G_1(t-t_1)\\
    \sum_{k=2}^K \beta_k G_1(t_k-t_2)  =G_1(t-t_2) \\
    \phantom{\sum_{k=2}^K \beta_k G_1(t_k-t_2) }  \vdots \\
    \beta_{K-1}+\beta_KG_1(t_{K}-t_{K-1})=G_1(t-t_{K-1})\\
    \beta_K=G_1(t-t_K).
  \end{cases}
\end{equation}
The infimum~\eqref{min_pb} is nonnegative if, and only if $\beta_1,\dots,\beta_K\ge 0$: the nonnegativity of  $\sum_{k=1}^Kx_kG(t-t_k)$ is clear when $\beta$'s are nonnegative from~\eqref{lin_to_opt}. Otherwise, if $\beta_j<0$ for some $j\in \{1,\dots, K\}$, the infimum is equal $-\infty$ by taking  $x_1,\dots,x_K$ such that $\sum_{k'=1}^k x_{k'}G(t_k-t_{k'})= \lambda \mathbf{1}_{k=j}$ with $\lambda>0$ arbitrary large (such $x$'s exists and is unique by linear independence of the linear functions $\sum_{k'=1}^k x_{k'}G(t_k-t_{k'})$ which is an invertible triangular system because $G(0)\not=0$).

Now, we calculate explicitly $\beta$'s values. We set $a_1=t-t_K$ and $a_k=t_{K-k+2}-t_{K-k+1}$ for $k\in \{2,\dots, K\}$, and we get
by the Gauss elimination method
\begin{equation*}
  \begin{cases}
    \sum_{k=1}^{K-1} \beta_k G_1(t_k-t_1)  =G_1(t-t_1)-G_1(a_1)G_1(t_K-t_1)=G_2(a_1,t_K-t_1)\\
    \sum_{k=2}^{K-1} \beta_k G_1(t_k-t_2)  =G_1(t-t_2)-G_1(a_1)G_1(t_K-t_2)=G_2(a_1,t_K-t_2) \\
    \phantom{\sum_{k=2}^K \beta_k G_1(t_k-t_2) }  \vdots \\
    \beta_{K-2}+\beta_{K-1}G_1(t_{K-1}-t_{K-2})=G_2(a_1,t_K-t_{K-2})\\
    \beta_{K-1}=G_1(t-t_{K-1})-\beta_K G_1(t_{K}-t_{K-1})=G_2(a_1,t_{K}-t_{K-1})=G_2(a_1,a_2) \\
    \beta_K=G_1(a_1).
  \end{cases}
\end{equation*}
Suppose that for $j\in \{2,\dots, K-1\}$, we have proven
\begin{equation*}
  \begin{cases}
    \sum_{k=1}^{K+1-j} \beta_k G_1(t_k-t_1)  =G_j(a_1,\dots,a_{j-1},t_{K+2-j}-t_1)\\
\phantom{ \sum_{k=K-(j+1)}^{K-j} \beta_k G_1(t_k-t_{K-(j+1)})} \vdots \\
   \sum_{k=K-j}^{K+1-j} \beta_k G_1(t_k-t_{K-j})  =G_j(a_1,\dots,a_{j-1},t_{K+2-j}-t_{K-j}) \\
    \beta_{K-(j-1)}=G_j(a_1,\dots,a_j) \\
    \phantom{\beta_{K-j}}\vdots\\
    \beta_K=G_1(a_1).
  \end{cases}
\end{equation*}
Then by using Gauss elimination, we get
\begin{equation*}
  \begin{cases}
    \sum_{k=1}^{K-j} \beta_k G_1(t_k-t_1)  =G_j(a_1,\dots,a_{j-1},t_{K+2-j}-t_1)-\beta_{K-(j-1)}G_1(t_{K+1-j}-t_1)\\
    \phantom{\sum_{k=1}^{K-j} \beta_k G_1(t_k-t_1)  =G_j(a_1,\dots,a_{j-1},t_{K+2-j}-t_1)}=G_{j+1}(a_0,\dots,a_{j},t_{K+1-j}-t_{1}) \\
\phantom{ \sum_{k=K-(j+1)}^{K-j} \beta_k G_0(t_k-t_{K-(j+1)})} \vdots \\
 \beta_{K-j}= G_j(a_1,\dots,a_{j-1},a_j+a_{j+1})- G_j(a_1,\dots,a_j) G_1(a_{j+1})=G_{j+1}(a_1,\dots,a_{j},a_{j+1})  \\
    \beta_{K-(j-1)}=G_j(a_1,\dots,a_j) \\
    \phantom{\beta_{K-j}}\vdots\\
    \beta_K=G_1(a_1),
  \end{cases}
\end{equation*}
which proves the induction step. Thus, we finally obtain
\begin{equation}
  \forall j \in \{1,\dots,K\}, \ \beta_{K+1-j}=G_j(a_1,\dots,a_j)
\end{equation}

We now prove the second part of the statement, and consider $x_1=1$, and $x_k=-\sum_{k'=1}^{k-1}G_1(t_k-t_{k'})x_{k'}$ for $k\ge 2$, so that $\sum_{k'=1}^kG(t_k-t_{k'})x_{k'}=0$. From~\eqref{lin_to_opt}, we get for $t=t_{K+1}$ that $x_{K+1}$ satisfies:
\begin{align*}
  G(0)x_{K+1}=-\sum_{k=1}^{K}G(t_{K+1}-t_{k})x_{k}&=-\sum_{k=1}^K \beta_k \left( \sum_{k'=1}^k x_{k'}G(t_k-t_{k'})\right) \\& = - \beta_1G(0)= -G(0)G_K(a_1,\dots,a_K),
\end{align*}
leading to $x_{K+1}=-G_K(t_{K+1}-t_{K},t_{K}-t_{K-1},\dots,t_2-t_1)$. The parameter $K\in \N^*$ being arbitrary, this gives the claim.
\end{proof}

We conclude this subsection by a useful property satisfied by nonincreasing kernels that preserves nonnegativity. 

\begin{prop}\label{prop_pos_gen}
  Let $G:\R_+\to \R_+$ be a nonincreasing kernel such that $G(0)>0$ that preserves nonnegativity. Let $K\in \N^*$, $0\le t_1 <\dots <t_K$ and $x_0,\dots,x_K \in \R$ be such that $x_0\ge 0$ and
  $$\forall k \in \{1,\dots, K\}, \ x_0+\sum_{k'=1}^kx_{k'}G(t_k-t_{k'})\ge 0.$$
  Then, we have $x_0+\sum_{k: t_k\le t} x_kG(t-t_k)\ge 0$ for all $t\ge 0$.
\end{prop}
\begin{proof}
  We introduce $\tilde{x}_1=-x_0/G(0)$ and $\tilde{x}_k=\frac{-1}{G(0)}\left( x_0+ \sum_{k'=1}^{k-1} \tilde{x}_{k'}G(t_k-t_{k'})\right)$, so that $ x_0+\sum_{k'=1}^k \tilde{x}_{k'}G(t_k-t_{k'})= 0$ for all $k\in \{1,\dots, K\}$.

  Let $\delta_k=x_k-\tilde{x}_k$ for $k\in \{1,\dots,K\}$. We have  $\forall k \in \{1,\dots, K\}, \ \sum_{k'=1}^k\delta_{k'}G(t_k-t_{k'})\ge 0$, and therefore $\sum_{k: t_k\le t} \delta_kG(t-t_k)\ge 0$ because $G$ preserves nonnegativity. Since $x_0+\sum_{k: t_k\le t} x_kG(t-t_k)=x_0+\sum_{k: t_k\le t} \tilde{x}_kG(t-t_k) + \sum_{k: t_k\le t} \delta_kG(t-t_k)$, it is then sufficient to check that $x_0+\sum_{k: t_k\le t} \tilde{x}_kG(t-t_k)\ge 0$ for all $t\ge 0$. This is clearly true for $t\in[0,t_1)$. We have $\tilde{x}_1\le 0$, which gives that $x_0+\tilde{x}_1G(t-t_1)\ge 0$ for $t\in[t_1,t_2)$ since it vanishes for $t=t_1$ and $G$ is nonincreasing. Then, we get that $\tilde{x}_2\le 0$. Again, this gives that $x_0+\tilde{x}_1G(t-t_1)+\tilde{x}_2G(t-t_2)\ge 0$ for $t\in[t_2,t_3)$ since it vanishes for $t=t_2$ and $G$ is nonincreasing. By iteration, we get that $\tilde{x}_k\le 0$ for all $k$ and $x_0+\sum_{k: t_k\le t} \tilde{x}_kG(t-t_k)\ge 0$ for all $t\ge 0$.
\end{proof}

\begin{remark}\label{rk_Hawkes1}
 We can straightforwardly extend Proposition~\ref{prop_pos_gen} by replacing the constant $x_0$ by a nonnegative nondecreasing function $f:\R_+\to \R_+$. Namely, we have $f(t)+\sum_{k: t_k\le t} x_kG(t-t_k)\ge 0$ for all $t\ge 0$ if $f(t_k)+\sum_{k'=1}^k x_{k'} G(t_k-t_{k'})\ge 0$ for all $k\ge 0$. As an application of this result, if $(N_t,t\ge 0)$ is a process that jumps with intensity 
 $$\lambda_t=f(t)+\int_0^tG(t-s) dN_s,$$
 and such that $dN_t\ge \frac{-1}{G(0)} \left(f(t_k)+ \int_{[0,t_k)}G(t-s) dN_s\right)$, then we have well $\lambda_t\ge 0$, for all $t\ge 0$ by Proposition~\ref{prop_pos_gen}.
\end{remark}
\begin{remark}\label{rk_rightcontinuity}
Let $G:\R_+\to \R_+$ be a nonincreasing kernel that preserves nonnegativity and is continuous in~$0$, i.e. $G(0)=G(0+)>0$. Then, $G$ is right continuous. Indeed, we get by Theorem~\ref{thm_char_pos} that $G_2(a_1,a_2)=G_1(a_1+a_2)-G_1(a_1)G_1(a_2)\ge 0$ for $a_1,a_2> 0$. Letting $a_2 \to 0$, we get $G(a_1+)\ge G(a_1)$ and thus  $G(a_1)=G(a_1+)$ since $G$ is nonincreasing.

As example, the kernel $G(t)=\alpha e^{-\rho t}+(1-\alpha)\mathbf{1}_{t=0}$ with $\alpha \in[0,1)$ and $\rho \ge 0$ satisfies $G_l(a_1,\dots,a_l)=\alpha(1-\alpha)^{l-1}e^{-\rho(a_1+\dots+a_l)}\ge 0$ and preserves nonnegativity by Theorem~\ref{thm_char_pos}.   
\end{remark}

\subsection{Completely monotone kernels preserve nonnegativity}

A function~$G:\R_+^*\to \R_+$ is said to be completely monotone if it is $\mathcal{C}^\infty$ on $\R_+^*$ and such that $(-1)^n G^{(n)}\ge 0$. It is well known from the celebrate Hausdorff-Bernstein-Widder theorem~\cite[Theorem IV.12a]{Widder} that this is equivalent to the existence of a Borel measure $\mu$ on $\R_+$ such that $G(t)=\int_{(0,+\infty)} e^{-\rho t} \mu(d\rho)$. When $G(0)<\infty$ (i.e. $\mu$ is finite), then $G$ can be extended to a function from $\R_+$ to $\R_+$, which we note $G:\R_+\to \R_+$.

\begin{theorem}\label{thm:positivite}
Complete monotone kernels  $G:\R_+\to \R_+$ preserve nonnegativity. More generally, if $\mu$ is a finite Borel measure on~$\R$ such that $G(t)=\int_{\R}e^{-\rho t}\mu(d\rho)<\infty$ for all $t\ge 0$, then $G$ preserves nonnegativity. \\ 
Besides, if $\mu$ is not a Dirac mass and if $x_1>0$, we have strict positivity in~\eqref{posit} for any $t>t_1$,  $t\not \in \{t_2,\dots,t_K\}$. 
\end{theorem}

 There is an abundant literature on completely monotone kernels and their properties that dates back to first half of the XXth century. We mention here the celebrated Bernstein theorem~\cite{Bernstein} and the article of Schoenberg~\cite[Theorem 3]{Schoenberg} who shows that a kernel~$G$ is completely monotone if, and only if $G(t^2)$ is positive definite (i.e. the matrix $(G((t_k-t_{k'})^2))_{1\le k,k'\le K}$ is positive semidefinite for any $K\in \N^*$ and any $t_1<\dots<t_K$). 
However, up to our knowledge,  it does not seem that the  property of preserving nonnegativity has been stated in the literature for completely monotone kernels.

\begin{proof}[Proof of Theorem~\ref{thm:positivite}]
By Example~\ref{example_expo}, Theorem~\ref{thm:positivite} is true for $\mu(d\rho)=\lambda \delta_{\rho_1}(d\rho)$. We can therefore assume without loss of generality that the support of $\mu$ is not a singleton, i.e. there is no $\rho_1 \in \R$ such that $\mu(\R)=\mu(\{\rho_1\})$.

By Theorem~\ref{thm_char_pos}, it is sufficient to check that the functions~$G_l$ are nonnegative. To do so, we use the second statement of Lemma~\ref{lem_char_pos}. We consider $(t_k)_{k\in \N^*}$ an increasing sequence of nonnegative real numbers, $x_1=1$ and  $x_k=\frac{-1}{G(0)} \sum_{k'=1}^{k-1} x_{k'}G(t_k-t_{k'})$ for $k\ge 2$. We are going to prove that $x_k\le 0$ for all~$k\ge 2$, which gives the nonnegativity of the functions $G_l$ since the sequence $(t_k)_{k\in \N^*}$ is arbitrary, and then the claim by the first statement of Lemma~\ref{lem_char_pos}.

 We define, for $\rho\in \R$, $X^\rho_t=\sum_{k: t_k\le t}x_ke^{-\rho(t-t_k)}$ and  $X_t=\int_{\R} X^\rho_t \mu(d\rho)$, so that $X_t=\sum_{k: t_k\le t}x_k G(t-t_k)$. In particular, we have $X^\rho_{t_k}=\sum_{j=1}^kx_je^{-\rho(t_k-t_j)}$ and $X^\rho_{t_k-}=\sum_{j=1}^{k-1}x_je^{-\rho(t_k-t_j)}$.

We now prove by induction on $k\ge 2$ the two following properties.
\begin{enumerate}
\item $x_j<0$ for $j\in \{2,\dots,k\}$.
\item Let $r_k$ be the unique~$\rho\in \R$ such that $\sum_{j=2}^k -x_j e^{\rho(t_j-t_1)}=1$. We have $\mu((-\infty,r_k) )>0$   and 
\begin{equation}\label{induc_rk}
  X^\rho_{t_k}>0 \text{ for } \rho<r_k\text{ and }X^{\rho}_{t_k}< 0 \text{ for }\rho > r_k.
\end{equation}
\end{enumerate}
For $k=2$, we observe that $X^\rho_{t_2-}=e^{-\rho(t_2-t_1)}$ is continuous decreasing with respect to~$\rho$ and $x_2=-\frac{\int_{\R} X^\rho_{t_2-} \mu(d\rho)}{\mu(\R_+)}<0$. We consider $r_2>0$ the unique real number such that $-x_2e^{r_2(t_2-t_1)}=1$.  It satisfies~\eqref{induc_rk} since  $X^\rho_{t_2}=X^\rho_{t_2-}+x_2=e^{-\rho(t_2-t_1)}-e^{-r_2(t_2-t_1)}$. Last, we notice that $X_{t_2}=0=\int_{\R}X^{\rho}_{t_2} \mu(d\rho)$ and thus $\int_{(-\infty,r_2)}X^{\rho}_{t_2} \mu(d\rho)= \int_{(r_2,+\infty)}-X^{\rho}_{t_2} \mu(d\rho)$. Since the support of $\mu$ is not a singleton, both integrals are necessarily positive and $\mu((-\infty,r_2))>0$.

We now assume that the induction hypothesis is true for $k$. Since $X_{t_{k+1}-}^\rho= X_{t_{k}}^\rho e^{-\rho (t_{k+1}-t_k)}$,  we have $X_{t_{k+1}-}^\rho>0$ for $\rho<r_k$ and $X_{t_{k+1}-}^\rho< 0$ for $\rho> r_k$. This gives $X_{t_{k+1}-}^\rho>X_{t_{k}}^\rho e^{-r_k(t_{k+1}-t_k)}$ for $\rho<r_k$ and $X_{t_{k+1}-}^\rho\ge X_{t_{k}}^\rho e^{-r_k(t_{k+1}-t_k)}$ for $\rho\ge r_k$.  We then get
\begin{align}
 - x_{k+1}&=\frac{\int_{\R} X_{t_{k+1}-}^\rho \mu(d \rho) }{\mu(\R)} \notag \\
  &>\frac{\int_{\R} X_{t_{k}}^\rho e^{-r_k(t_{k+1}-t_k)} \mu(d \rho) }{\mu(\R)}=\frac{e^{-r_k(t_{k+1}-t_k)}}{\mu(\R)} X_{t_k}=0, \label{strict_ineg_pos}
\end{align}
the inequality being strict since $\mu((-\infty,r_k))>0$. From~$X^{\rho}_{t_{k+1}}=\sum_{j=1}^{k+1}x_je^{-\rho(t_{k+1}-t_j)}$, we get 
  $$X^{\rho}_{t_{k+1}}>0 \iff 1 > \sum_{j=2}^{k+1} (-x_j)e^{\rho(t_j-t_1)} \iff \rho<r_{k+1},$$
 where $r_{k+1}$ is such that $\sum_{j=2}^{k+1} (-x_j)e^{\rho(t_j-t_1)}=1$, since the map $\rho \mapsto \sum_{j=2}^{k+1} (-x_j)e^{\rho(t_j-t_1)} $ is an increasing bijection from $\R$ to $\R_+^*$. This gives~\eqref{induc_rk}. From
  $\int_{\R} X_{t_{k+1}}^\rho \mu(d \rho)=0$, we get  $\int_{(-\infty,r_{k+1})}X^{\rho}_{t_{k+1}} \mu(d\rho)= \int_{(r_{k+1},+\infty)}-X^{\rho}_{t_{k+1}} \mu(d\rho)>0$ since $\mu$ does not charge only one point, and therefore $\mu((-\infty,r_{k+1}))>0$.

  We now prove the last statement of the theorem (strict inequality). First, we observe that a slight adaptation of~\eqref{strict_ineg_pos} gives $X_t>\frac{e^{-r_k(t-t_k)}}{\mu(\R)}X_{t_k}=0$ for $t \in (t_k,t_{k+1})$. Now, let us consider $y_1>0$ and $y_2,\dots,y_K$ be such that $\sum_{k'=1}^ky_{k'}G(t_k-t_{k'})\ge 0$ for all $k\in \{1,\dots,K\}$. Then, $\delta_k=y_k-y_1 x_k$ satisfies $\sum_{k'=1}^k\delta_{k'}G(t_k-t_{k'})\ge 0$ for all $k\in \{1,\dots,K\}$ since $\delta_1=0$ and $\sum_{k'=1}^k x_{k'}G(t_k-t_{k'})= 0$ for $k\ge 2$. This shows $\sum_{k: t_k\le t } \delta_k G(t-t_k)\ge 0$ and thus $\sum_{k: t_k\le t } y_k G(t-t_k)\ge y_1 X_t$ that is strictly positive for $t>t_1, t \not \in \{t_2,\dots,t_K\}$. 
\end{proof}

\begin{remark}
  Let $G(t)=2e^{-t}-e^{-2t}$ for $t\ge 0$. This is a nonnegative decreasing kernel that clearly does not satisfy the assumptions of Theorem~\ref{thm:positivite}. It does not preserve nonnegativity by Theorem~\ref{thm_char_pos} since for $a_1,a_2>0$
$$G_2(a_1,a_2)=G(a_1+a_2)-G(a_1)G(a_2)=-2e^{-(a_1+a_2)}(1-e^{-a_1})(1-e^{-a_2})<0.$$ 
\end{remark}

\begin{remark}\label{rk_Hawkes2}
  ("A process that has been initiated but cannot be switched off")
  Let us consider a controlled Hawkes process $(N_t,t\ge 0)$ with unit jumps and intensity $\lambda_t=\int_0^t G(t-s)[dN_s+dJ_s]$, where $J_t$ is an adapted control $J_t=\sum_{i=0}^\infty {X_{i}}\mathbf{1}_{i\le t}$ with $X_0>0$. More precisely, if $(\cF_t)_{t\ge 0}$ denotes the natural filtration associated to $N$,  $X_{i}$ is supposed to be $\cF_{i-}$-measurable and such that $G(0)X_{i}+\int_{[0,i)} G(i-s)[dN_s+dJ_s]\ge 0$. If $G(t)=e^{-\rho t}$ for some $\rho\ge 0$, then the proces $N$ can be switched off at time $i$ by taking the control $X_i=-\int_{[0,i)} G(i-s)[dN_s+dJ_s]$ and $X_j=0$ for $j>i$. We have indeed $\lambda_t=0$ for $t\ge i$ in this case. Instead, if the kernel~$G$ is completely monotone but not proportional to an exponential function, then $\lambda_t>0$ for $t \not \in \N$ by using the strict positivity of~\eqref{posit} given by Theorem~\ref{thm:positivite}.  
\end{remark}

We conclude this section with two remarks that make some connections between Theorem~\ref{thm:positivite} and other problems in the literature.

\begin{remark}\label{rk_Schur} Let us consider $G:\R_+\to \R_+$ a completely monotone kernel such that $G(0)=1$, so that $G(t)=\E[e^{-tR}]$ for some nonnegative random variable $R$. We set $X=e^{-R}$ that takes values in $(0,1]$, and consider $(X_i)_{i\ge 1}$ that are i.i.d. with the same law as~$X$. Since $G$ preserves nonnegativity by Theorem~\ref{thm:positivite}, we have $G_l\ge 0$ according to Theorem~\ref{thm_char_pos}. The condition $G_2\ge 0$ is equivalent to have  $\E[X_1^{a_1+a_2}]\ge \E[X_1^{a_1} X_2^{a_2}]$ for $a_1,a_2\ge 0$, which is a classical comonotonic inequality, see e.g. McNeil et al.~\cite[Theorem 5.25]{McNeil}. It can also be seen as a consequence of the Schur convexity of $\phi(a_1,a_2)=\E[X_1^{a_1} X_2^{a_2}]$, see e.g. Marshall et al.~\cite[G.2.h. p.~126]{MOA}. For $l=3$, the inequality $G_3\ge 0$ can be rewritten as
      \begin{equation}\label{ineq_G3}\E[X_1^{a_1+a_2+a_3}]+\E[X_1^{a_1}]\E[X_1^{a_2}]\E[X_1^{a_3}]\ge \E[X_1^{a_1}X_2^{a_2+a_3}]+\E[X_1^{a_1+a_2}X_2^{a_3}], \ a_1,a_2,a_3\ge 0,
      \end{equation}
and does not seem to be proved easily from comonotonicity or Schur convexity arguments. Thus, Theorem~\ref{thm:positivite} which gives the nonnegativity of the functions $G_l$, $l\ge 2$, gives  a series of inequalities like~\eqref{ineq_G3} that may have their own interest.  
\end{remark}

\begin{remark}\label{rk_Bartholomew}Let $n\ge 2$. Bartholomew~\cite{Bartho} has given sufficient conditions under which the linear combination of exponential functions $f(t)=\sum_{i=1}^n\mu_i e^{-\rho_it}$ with $\mu_1,\dots,\mu_n \in \R^*$ and $0<\rho_1<\dots<\rho_n$ is nonnegative for all $t\ge 0$. To be precise, \cite{Bartho} considers in addition that $\mu_i=p_i \rho_i$ with $\sum_{i=1}^np_i =1$ to get then a probability distribution, but this does not change the problem up to a scaling factor. Here, Theorem~\ref{thm:positivite} gives another sufficient condition: $f(t)\ge 0$ for all $t\ge 0$ if there exists $\gamma_1,\dots,\gamma_n>0$, $K\in \N^*$, $x_1,\dots,x_K \in \R$ and $0\le t_1<\dots<t_K$ such that $\mu_i=\gamma_i \sum_{k=1}^K x_k e^{-\rho_i(t_{K}-t_k)}$ and
  $$\forall k \in \{1,\dots,K\},\ \sum_{i=1}^n\gamma_i \sum_{k'=1}^k x_k e^{-\rho_i (t_k-t_{k'})}\ge 0. $$
Indeed, considering the completely monotone function $G(t)=\sum_{i=1}^n \gamma_i e^{-\rho_it}$, we get $\sum_{k=1}^K x_kG(t-t_k) > 0$ for $t\ge t_K$ by Theorem~\ref{thm:positivite}. Therefore, the function $f(t)=\sum_{k=1}^K x_kG(t+t_K-t_k)=\sum_{i=1}^n\mu_i e^{-\rho_it}$
is positive on $\R_+^*$. Naturally, one may wonder if this approach may lead to a necessary and sufficient condition. This is left for further research. 
\end{remark}

\subsection{Characterization with the resolvent of the first kind}

We now consider a nonnegative measurable kernel $G:\R_+ \to \R_+$ that is locally integrable and note $G\in L^1_{loc}(\R_+,\R_+)$. The resolvent of the first kind of~$G$ is a signed Borel measure~$R$ on~$\R_+$ with locally bounded total variation (i.e. $|R|([0,t])<\infty$ for all $t>0$) such that 
\begin{equation}\label{def_resolvent}
  \forall  t\ge 0, \int_{[0,t]}G(t-s)R(ds)=1, 
\end{equation}  
see Gripenberg et al.~\cite[Definition 5.5.1]{Gripenberg}. Such a resolvent does not necessarily exists, but it is necessarily unique by~\cite[Theorem 5.5.2]{Gripenberg}. When $G$ is nonincreasing and not identically zero, the resolvent~$R$ exists and is a nonnegative measure~\cite[Theorem 5.5.5]{Gripenberg}. Note that kernels are often used to describe phenomena that fade away along the time, and therefore nonincreasing kernels are of practical relevance. 

\begin{definition}\label{def_nonincr}
  A $\R$-valued measure~$M$ on~$\R_+$ with locally bounded total variation  is said to be nonincreasing if for any $0\le h_1\le h_2$, $M(h_1+ds)-M(h_2+ds)$ is a nonnegative measure. 
\end{definition}
\noindent Here, we note $M(h+ds)$ the pushforward signed  measure obtained with the map $s\mapsto s-h$, i.e. $\int f(s)M(h+ds)=\int f (s-h) M(ds)$ for $f$ bounded measurable with compact support.

The property of having a nonincreasing resolvent has been put in evidence recently by Abi Jaber et al.~\cite[Theorem 3.6]{AJLP} to get the nonnegativity of some Stochastic Volterra Equations. However, this property appears as a sufficient condition to guarantee the nonnegativity. By the next theorem, we see that this is a necessary and sufficient condition for a nonincreasing kernel to preserve nonnegativity.

\begin{theorem}\label{thm_carac_np}
Let $G :\R_+ \to \R_+$  be a  right-continuous kernel with $G(0)>0$. 

\begin{enumerate}
  \item  Let us assume that $G$ has a resolvent of the first kind~$R$, and that $R$ is nonincreasing. Then, $G$ preserves nonnegativity. 
  \item Let us assume that $G$ is nonincreasing. The kernel $G$ preserves nonnegativity if, and only if its resolvent $R$ is nonincreasing. Besides, we have in this case 
$$R(ds)=\frac{1}{G(0)}\delta_0(ds)+r(s)ds,$$
where $r:\R^*_+\to \R_+$ is a locally integrable nonincreasing function. 
\end{enumerate}
\end{theorem}
\begin{remark}
Let us note that if $G:\R_+\to \R_+$ is nonincreasing, the set $Disc=\{ t\ge 0, \lim_{s\to t, s<t} G(\max(s,0))>\lim_{s\to t, s>t} G(s)\}$ is countable. Therefore, there exists a right-continuous modification of~$G$ that coincides with~$G$ on $\R_+ \setminus Disc$. This modification has then the same Laplace transform as~$G$ and thus the same resolvent, see~\cite[Theorem 5.5.5]{Gripenberg}. 
\end{remark}

\begin{remark}
When $G:\R_+\to \R_+$ is completely monotone, it has a resolvent that is the sum of a point mass at zero and a completely monotone function~\cite[Theorem 5.5.4]{Gripenberg}. Thus, Theorem~\ref{thm_carac_np} gives also that completely monotone kernels preserves nonnegativity. Theorem~\ref{thm:positivite} is however a bit more precise for those kernels and give strict positivity in~\eqref{posit} when the kernel is not a single exponential. Besides, Theorem~\ref{thm:positivite} also works for kernels $G(t)=\int_{\R} e^{-\rho t} \mu(d\rho)$ with $\mu$ being a finite measure on~$\R$ (not only $\R_+$). 
\end{remark}

\begin{proof}
 (1) We assume that the resolvent $R$ of~$G$ exists and is nonincreasing. We prove by induction on $K\in \N^*$ that if 
   $x_1,\dots,x_K \in \R$ and $0\le t_1< \dots<t_K$ are such that
  $$\forall k \in \{1, \dots, K\},\ \sum_{k'=1}^k x_{k'}G(t_k-t_{k'})\ge 0, $$
  then, we have $ \forall t \ge 0,\ \psi_K(t)\ge 0$ with 
$ \psi_K(t)= \sum_{k :  t_k \le t} x_k G(t-t_k) \ge 0 $. For $K=1$, this is obvious. Let us assume $K\ge 2$ and that the result holds true for any $k\le K-1$. Thus, the function $\psi_K$ is nonnegative on~$[0,t_K]$. Let $t\ge t_K$. We have 
$$\int_{[0,t]} \psi_K(t-s) R(ds)=\sum_{k=1}^K x_k \int_{[0,t-t_k]}G(t-t_k-s)R(ds)= \sum_{k=1}^K x_k,$$
by using~\eqref{def_resolvent}.
Therefore, we have $\int_{[0,t_K+h]} \psi_K(t_K+h-s) R(ds)=\int_{[0,t_K]} \psi_K(t_K-s) R(ds)$ for any $h\ge 0$. By a change of variable, we get
$$\int_{[0,t_K+h]} \psi_K(t_K+h-s) R(ds)=\int_{[0,h)} \psi_K(t_K+h-s) R(ds)+\int_{[0,t_K]} \psi_K(t_K-s) R(h+ds),$$
and define
$$\varphi(h):=\int_{[0,h)} \psi_K(t_K+h-s) R(ds)=\int_{[0,t_K]} \psi_K(t_K-s) [R(ds)-R(h+ds)].$$
This function is nonnegative and nondecreasing since $\psi_K$ is nonnegative on $[0,t_K]$ and $R$ is nonincreasing.
We have by using Fubini theorem and a change of variable:
\begin{align*}
\int_{[0,h]} \varphi(v)G(h-v) dv&=\int_{[0,h]}\int_{[0,h]} \mathbf{1}_{s<v}G\left((h-s)-(v-s)\right)\psi_K(t_K+v-s) R(ds) dv  \\
&= \int \int \mathbf{1}_{0<v'\le h-s}\mathbf{1}_{0\le s\le h}G\left((h-s)-v'\right)\psi_K(t_K+v') R(ds) dv' \\
&=\int \int \mathbf{1}_{0<v'\le h}\mathbf{1}_{0\le s\le h-v'} G\left((h-v')-s\right)\psi_K(t_K+v') R(ds) dv'\\
&=\int_0^h \psi_K(t_K+v')dv',
\end{align*} 
where the last equality comes from the resolvent property. We note that $v\mapsto \psi_K(t_K+v)$ is right-continuous since  $G$ is right-continuous. Besides, since $G$ is nonnegative and $\varphi$ is nonnegative and nondecreasing, we get that $h\mapsto \int_{[0,h]} \varphi(v)G(h-v) dv$ is nondecreasing (indeed, we have $\int_{[0,h+\delta]} \varphi(v)G(h+\delta-v) dv \ge \int_{[0,h]} \varphi(v+\delta)G(h-v) dv\ge  \int_{[0,h]} \varphi(v)G(h-v) dv$ for $\delta\ge 0$). This gives that $\psi_K(t_K+h)$ is nonnegative for all $h>0$ by taking the right-derivative.

(2) We now assume that $G:\R_+\to \R_+$ is nonincreasing. It thus has a resolvent of the first kind~$R$ by~\cite[Theorem 5.5.5]{Gripenberg}. The sufficient condition has been proven above (part~(1)). We now assume that $G$ preserves nonnegativity, and show that $R$ is nonincreasing. Without loss of generality, we can assume that $G(0)=1$.  For $n\in \N^*$, 
we consider the time grid $\{i/n, i \in \N\}$ and define $x^n_0=1$ and for $k\ge 1$,
$$ x^n_{k}=1-\sum_{i=0}^{k-1}x^n_iG\left( \frac{k-i}n\right),$$
so that we have $\sum_{i=0}^{k}x^n_iG\left( \frac{k-i}n\right)=1$ for all $k\in \N$. 
We claim that $(x^n_k)_{k\ge 0}$ is a nonnegative nonincreasing sequence. We prove that $x^n_k\ge 0$ by induction on $k$. This is true for $k=0$. If $x^n_i\ge 0$ for $0\le i\le k-1$, then we have $\sum_{i=0}^{k-1}x^n_iG\left( \frac{k-i}n\right)\le \sum_{i=0}^{k-1}x^n_iG\left( \frac{k-1-i}n\right)=1$ since $G$ is nonincreasing, which proves that $x^n_k\ge 0$. We now show by induction that $x^n_k-x^n_{k+1}=G_{k+1}\left(\frac 1n,\dots, \frac 1n \right)$, where $G_k$ is defined by~\eqref{def_Gl}. We have $x^n_0-x^n_1=1-(1-G(1/n))=G(1/n)$, so the claim is true for $k=0$. We now consider $k\in \N^*$ and suppose that the property is true for all $i\le k-1$. Then, we have 
\begin{align*}
  x^n_k-x^n_{k+1}&=\sum_{i=0}^{k} x^n_iG\left( \frac{k+1-i}n\right) -\sum_{i=0}^{k-1}x^n_iG\left( \frac{k-i}n\right)\\
&=G\left( \frac{k+1}n\right)+\sum_{i=0}^{k-1}(x^n_{i+1}-x^n_i)G\left( \frac{k-i}n\right)\\&=G\left( \frac{k+1}n\right)-\sum_{i=0}^{k-1}G_{i+1}\left(\frac 1n,\dots, \frac 1n \right) G\left( \frac{k-i}n\right),
\end{align*}  
by using the induction hypothesis. 
On the other hand, we get from~\eqref{def_Gl} that
\begin{align*}
  G_{k+1}\left(\frac 1n,\dots, \frac 1n \right)&=G_{k}\left(\frac 1n,\dots, \frac 1n,\frac 2n \right)-G\left( \frac 1n \right)G_{k}\left(\frac 1n,\dots, \frac 1n \right)\\
&  =G_{k-1}\left(\frac 1n,\dots, \frac 1n,\frac 3n \right)-G\left( \frac 2n \right)G_{k-1}\left(\frac 1n,\dots, \frac 1n \right)-G\left( \frac 1n \right)G_{k}\left(\frac 1n,\dots, \frac 1n \right)\\
&  =G\left( \frac{k+1}n\right)-\sum_{i=0}^{k-1}G_{i+1}\left(\frac 1n,\dots, \frac 1n \right) G\left( \frac{k-i}n\right),
\end{align*}
by iteration. Therefore, $x^n_k-x^n_{k+1}=G_{k+1}\left(\frac 1n,\dots, \frac 1n \right)$, which shows that the sequence is nonincreasing by Theorem~\ref{thm_char_pos}. 

We now define the measure $R^n(ds)=\sum_{k\ge 0} x^n_k \delta_{\frac kn}(ds)$.
For $t>0$, we have $\int_{[0,t]}G(t-s)R^n(ds)=\sum_{i=0}^{\lfloor nt \rfloor} G(t-\frac in)x^n_i$. Since $G$ is nonincreasing, we get
\begin{align}\label{intermed}1-x^n_{\lfloor nt \rfloor+1} &=\sum_{i=0}^{\lfloor nt \rfloor} G\left(\frac {\lfloor nt \rfloor +1-i}n\right)x^n_i \\
  &\le \int_{[0,t]}G(t-s)R^n(ds) \le \sum_{i=0}^{\lfloor nt \rfloor} G\left(\frac {\lfloor nt \rfloor -i}n\right)x^n_i=1.\notag
\end{align} 
By right continuity at~$0$, we can find $\epsilon>0$ arbitrary small such that $G(\epsilon)>0$ and get from the last inequality $1\ge G(\epsilon)R^n([0,\epsilon])\ge G(\epsilon)\lfloor n \epsilon  \rfloor x^n_{\lfloor n \epsilon  \rfloor}$. This gives on the one hand that $x^n_{[n\epsilon]}\to_{n\to \infty} 0$ and thus  $x^n_{[nt]}\to_{n\to \infty} 0$ for any $t>0$  since $(x^n_k)_{k\ge 0}$ is nonnegative nonincreasing. We get then 
\begin{equation}\label{conv_convol}
   \int_{[0,t]}G(t-s)R^n(ds) \underset{n \to \infty}\to 1.
\end{equation}
On the other hand, we have that for any $t\ge\epsilon$, 
$$R^n([0,t])\le R^n([0,\epsilon])+\sum_{i=\lfloor n \epsilon  \rfloor +1}^{\lfloor n t  \rfloor} x^n_i \le \frac 1 {G(\epsilon)}+ (\lfloor n t \rfloor- \lfloor n \epsilon  \rfloor) x_{\lfloor n \epsilon  \rfloor}\le \frac 1 {G(\epsilon)}\left(1+\frac{\lfloor n t \rfloor- \lfloor n \epsilon  \rfloor}{ \lfloor n \epsilon  \rfloor}\right),$$
is uniformly bounded in~$n$. Therefore, by using Cantor's diagonal argument, we can find a subsequence $R^{\varphi(n)}$ that weakly converges to a Borel measure $R^\infty$ on $\R_+$ for all intervals $[0,m]$, $m\in \N^*$. We have $1= \lim_{n\to \infty}R^{\varphi(n)}(\{0\}) \le R^\infty(\{0\})\le R^\infty([0,\epsilon))\le  \liminf_{n\to \infty}R^{\varphi(n)}([0,\epsilon))\le \frac 1{G(\epsilon)}$ and thus $R^\infty(\{0\})=1$ since $\epsilon$ is arbitrary small and $G(0+)=1$. Besides, 
using the same argument as above, we get for $t_2>t_1\ge\epsilon$,
$$R^\infty((t_1,t_2))\le \liminf_{n\to \infty}R^{\varphi(n)}((t_1,t_2))\le \frac{1}{G(\epsilon)} \liminf_{n\to \infty} \frac{\lfloor n t_2 \rfloor- \lfloor n t_1  \rfloor}{ \lfloor n \epsilon  \rfloor}=\frac{t_2-t_1}{\epsilon G(\epsilon)}.$$ 
This shows that $R^\infty$ is absolutely continuous with respect to the Lebesgue measure on $[\epsilon,+\infty)$, for any $\epsilon>0$. By the Radon-Nikodym theorem, this gives the existence of a nonnegative locally integrable function $r:\R_+^* \to \R$ such that 
$R^\infty(ds)=\delta_0(ds)+r(s)ds$.
Since $G$ is bounded, continuous at~$0$ and with a countable set of discontinuities, we get from~\eqref{conv_convol} that $\int_{[0,t]}G(t-s)R^\infty(ds)=1$ for all $t\ge 0$. Therefore, $R^\infty=R$ by~\cite[Theorem 5.5.2]{Gripenberg}. It remains to show that $r$ is nonincreasing. Let $\psi:\R_+\to\R_+$ be continuous with compact support. 
Since $(x^n_i)_{i\ge 0}$ is nonincreasing, we get $\int_{t^n_1}^\infty \psi(s-t^n_1)R^n(ds)=\sum_{i=0}^\infty \psi(i/n) x^n_{i+\lfloor nt_1 \rfloor}\ge\sum_{i=0}^\infty \psi(i/n) x^n_{i+\lfloor nt_2 \rfloor}= \int_{t^n_2}^\infty \psi(s-t^n_2)R^n(ds)$ for $0<t_1\le t_2$ and $t^n_j=\frac{\lfloor n t_j \rfloor}n$. Letting $n\to \infty$, this leads to $\int_0^\infty \psi(s)( r(t_1+s)-r(t_2+s))ds\ge 0$ and thus $ r(t_1+s)-r(t_2+s)$ is $ds$-a.e. nonnegative. 
\end{proof}

\begin{example}
We give here examples of nonnegativity preserving kernels that are not completely monotone. The kernel $G(t)=\frac{1}5\left(2+(3 \cos(t)+\sin(t) )e^{-2t}\right)$ is nonnegative and has the resolvent $R(ds)=\delta_0(ds) +(2e^{-s}-e^{-2s})ds$, this can be checked by using their Laplace transform. Since $s\mapsto 2e^{-s}-e^{-2s}$ is nonincreasing, this gives by Theorem~\ref{thm_carac_np} that $G$ preserves nonnegativity. We note that $G'(t)=-(\cos(t)+\sin(t))e^{-2t}$, and therefore $G$ is not monotone. The kernel $\tilde{G}(t)=e^{-5t}G(t)$ also preserves nonnegativity by Corollary~\ref{cor_expoinvariance}. It is not completely monotone and nonincreasing since  $\tilde{G}'(t)=-\left(2+(2\sin(t)+4\cos(t))e^{-2t}\right)e^{-5t}\le 0$ for $t\ge 0$. 
\end{example}


\section{Stochastic Volterra Equations on a closed convex domain}\label{Sec_SVEs}

\subsection{An approximation of Stochastic Volterra Equations by splitting}\label{Subsec_SVEapprox}

Let $d\in \N^*$, $W$ be a $d$-dimensional Brownian motion and $\cF=(\cF_t)_{t\ge 0}$ the associated filtration. Let $\cM_d(\R)$ be the set of real square $d\times d$ matrices. We consider a continuous kernel and $G:\R_+ \to \R$ and we note
\begin{equation}\label{modulus}
  \omega_{G,T}(\delta)=\max_{s,t \in [0,T]: |s-t|\le \delta}|G(s)-G(t)|
\end{equation} its modulus of continuity on~$[0,T]$. Let us note that we do not need to assume that $G$ is nonnegative for the approximation results. We consider coefficients $b:\R^d\to \R^d$ and $\sigma:\R^d \to \cM_d(\R)$  that are Lipschitz continuous:
\begin{equation}\label{Lip}
  |b(x)-b(y)|+|\sigma(x)-\sigma(y)|\le L |x-y|,  \tag{{\bf Lip}}
\end{equation}
which gives in particular the sublinearity property
\begin{align}\label{SubLin}
  |b(x)|^2+ |\sigma(x)|^2 \le L'(1+|x|^2)  \tag{{\bf SubLin}} ,
\end{align}
with $L'=\max(2|b(0)|^2+2|\sigma(0)|^2,4 L^2)$. We consider the following Stochastic Volterra Equation
\begin{equation}\label{SVE}
  X_t=x_0+\int_0^tG(t-s)b(X_s)ds +\int_0^t G(t-s) \sigma(X_s)dW_s, \ t\ge 0,
\end{equation}
that admits by~\cite[Theorem 3.A]{BeMi} a unique solution $X$ that is adapted to~$\cF$ and such that $\sup_{t\in[0,T]}\E[|X_t|^2]<\infty$ for all $T>0$. 

Let $T>0$, $N\in\N^*$ and $t_k=kT/N$, for $k\in\{0,\dots,N\}$. We consider the following approximation scheme~$\hat{X}$ that consists in splitting, on each time step, the kernel convolution and the diffusion. Thus, we define $\hat{X}_t=x_0$ for $t\in [t_0,t_1)$ and $(\xi_{t},t\in[t_0,t_1))$ as the unique strong solution of the following Stochastic Differential Equation (SDE)
  \begin{equation}\label{SDE_first_interval} \xi_{t}=x_0+\int_{t_0}^t G(0) b(\xi_s)ds+ \int_{t_0}^t G(0) \sigma(\xi_s)dW_s, t\in[t_0,t_1),
  \end{equation}
    and we set $\hat{X}_{t_1}= \xi_{t_1-}=x_0+( \xi_{t_1-}-x_0)$. Then, we define
    $$ \hat{X}_{t}= x_0+\frac{\xi_{t_1-}-x_0}{G(0)}G(t-t_1),\ t \in [t_1,t_2).$$
We note that $\hat{X}_{t_2-}=x_0+\frac{\xi_{t_1-}-x_0}{G(0)}G(t_2-t_1) $  is $\mathcal{F}_{t_1}$-adapted and consider $(\xi_{t},t\in[t_1,t_2))$, the unique strong solution of the following SDE
  $$ \xi_{t}=\hat{X}_{t_2-}+\int_{t_1}^t G(0) b(\xi_s)ds+ \int_{t_1}^t G(0) \sigma(\xi_s)dW_s, t\in[t_1,t_2).$$
    We then set $\hat{X}_{t_2}= \xi_{t_2-}=x_0+\frac{\xi_{t_1-}-x_0}{G(0)}G(t_2-t_1)+(\xi_{t_2}-\hat{X}_{t_2-})$. 
      
Suppose that we have constructed $(\hat{X}_{t},t\in [t_0,t_k])$, for some $k<N$. Then, we define $(\hat{X}_{t},t\in (t_k,t_{k+1}])$ as follows. We first set 
\begin{equation*}
  \hat{X}_{t}= x_0+\sum_{j=1}^k\frac{\hat{X}_{t_{j}}-\hat{X}_{t_j-}}{G(0)}G(t-t_j),\ t \in (t_k,t_{k+1})
\end{equation*}
and have $\hat{X}_{t_{k+1}-}=x_0+\sum_{j=1}^k\frac{\hat{X}_{t_{j}}-\hat{X}_{t_j-}}{G(0)}G(t_{k+1}-t_j)$ that is $\cF_{t_k}$-measurable. Then, we consider $(\xi_{t},t\in[t_k,t_{k+1}))$, the unique strong solution of the following SDE
\begin{equation}\label{def_xit} \xi_{t}=\hat{X}_{t_{k+1}-}+\int_{t_{k}}^t G(0) b(\xi_s)ds+ \int_{t_{k}}^t G(0) \sigma(\xi_s)dW_s, t\in[t_{k},t_{k+1}),
\end{equation}
and define $\hat{X}_{t_{k+1}}=\xi_{t_{k+1}-}$. 

We have thus defined two processes $(\hat{X}_t,t\in [0,T])$ and $(\xi_t,t\in[0,T])$ that are piecewise continuous and càdlàg. We note that 
$$ \frac{\hat{X}_{t_{k}}-\hat{X}_{t_k-}}{G(0)}=\int_{t_{k-1}}^{t_k}[b(\xi_s)ds+\sigma(\xi_s)dW_s],$$
so that we can rewrite $\hat{X}_t$ as follows for $t \in [0,T]$:
\begin{align}
  \hat{X}_{t}&=x_0+\sum_{k: t_k\le t} G(t-t_k) \int_{t_{k-1}}^{t_{k}}[b(\xi_s)ds+ \sigma(\xi_s)dW_s] \notag\\&=x_0+ \int_0^{t_{\eta(t)}}G(t-t_{\eta(s)+1})[b(\xi_s)ds+ \sigma(\xi_s)dW_s], \label{hat_X}
\end{align}
where $\eta(t)=k$ for $t\in[t_k,t_{k+1})$ and $k\in \N$, so that $t_{\eta(t)}\le t <t_{\eta(t)+1}$.



\begin{lemma}\label{lem_apriori_estimate}
  Let $T>0$. We assume~\eqref{SubLin} and that $G:\R_+\to \R$ is a continuous function with modulus of continuity~\eqref{modulus} on~[0,T]. 
  There exists a constant $C\in\R_+^*$ that only  depends on  $T$, $ \max_{t\in [0,T]} |G(t)|$, $L'$ and $|x_0|$ such that \blue{for all $N\in \N^*$,} $$\sup_{t\in [0,T]}\E[|\xi_t|^2+|\hat{X}_t|^2]\le C \text{ and }\sup_{t\in [0,T]}\E[|\xi_t-\hat{X}_t|^2]\le C \left(\frac{T}{N} +\omega^2_{G,T}\left( \frac T N\right)\right).$$
\end{lemma}
\begin{proof}
  Let $k\in\{0,\dots,N-1\}$ and $t\in [t_k,t_{k+1})$. We get from~\eqref{def_xit} that
    $$|\xi_t|^2 \le 3 |\hat{X}_{t_{k+1}-}|^2+3G(0)^2 \left|\int_{t_k}^tb(\xi_s)ds \right|^2+3G(0)^2 \left|\int_{t_k}^t\sigma(\xi_s)dW_s \right|^2.$$
    By Jensen inequality, Itô isometry and~\eqref{SubLin}, we get
    $$\E[|\xi_t|^2|\cF_{t_k}]\le  3 |\hat{X}_{t_{k+1}-}|^2+3G(0)^2(1\vee T)L' \int_{t_k}^t (1+\E[|\xi_s|^2|\cF_{t_k}])ds.$$
    By Gronwall Lemma, this gives \begin{equation}\label{intermed_xi}
        \E[|\xi_t|^2]\le K(1+\E[|\hat{X}_{t_{k+1}-}|^2])e^{K(t-t_k)}\text{ for }t\in[t_k,t_{k+1}),
      \end{equation}
where~$K\in \R_+$ is a constant that depends on $T$, $|G(0)|$, and $L'$. We then get by iteration on~$k\in \{1,\dots,N\}$ that $\sup_{t\in[0,t_k]}\E[|\xi_t|^2]+\E[|\hat{X}_t|^2]<\infty$. However, this estimate is not sharp enough to get a bound that does not depend on~$N$. To get this bound, we apply Jensen inequality and Itô isometry to~\eqref{hat_X} for $t\to t_{k+1}$, $t<t_{k+1}$, and get
        \begin{align*}\E[|\hat{X}_{t_{k+1}-}|^2]&\le 3|x_0|^2+3\int_0^{t_{k}}G^2(t_{k+1}-t_{\eta(s)+1})[T\E[|b(\xi_s)|^2]+\E[|\sigma(\xi_s)|^2]]ds\\
          &\le 3|x_0|^2+3 (T\vee 1) L' \max_{t\in [0,T]} G^2(t)\int_0^{t_{k}} 1+\E[|\xi_s|^2]ds,
        \end{align*}
        where we have used~\eqref{SubLin} for the second inequality. We now plug the inequality~\eqref{intermed_xi} to get the existence of a constant $C$ that only depends on $T$, $ \max_{t\in [0,T]} |G(t)|$, $L'$ and $|x_0|$ such that
$$\E[|\hat{X}_{t_{k+1}-}|^2]\le C+\frac{CT}{N} \sum_{j=1}^{k}\E[|\hat{X}_{t_j-}|^2].$$
We can then apply the discrete Gronwall lemma (see e.g. Clark~\cite{Clark}) to get that $\sup_{1\le k\le N}\E[|\hat{X}_{t_k-}|^2]\le Ce^{CT}$. Combining this with~\eqref{intermed_xi}, we get the first claim.

We now prove the second claim. For $t\in[t_k,t_{k+1})$, we use the following inequality:
  $$ |\hat{X}_t-\xi_t|^2\le 2 |\hat{X}_t-\hat{X}_{t_{k+1}-}|^2 +2 |\hat{X}_{t_{k+1}-}-\xi_t|^2. $$
On the one hand, we have by~\eqref{def_xit}, Jensen inequality, Itô isometry, \blue{the growth assumption~\eqref{SubLin} and Lemma~\ref{lem_apriori_estimate}}
    \begin{align}\E[|\xi_t-\hat{X}_{t_{k+1}-}|^2]&\le 2  \int_{t_k}^t G(0)^2(T\E[|b(\xi_s)|^2]+\E[|\sigma(\xi_s)|^2]ds \notag\\
      &\le 2(T\vee 1)G(0)^2 L' \int_{t_k}^t 1+ \E[|\xi_s|^2]ds \le C\frac{T}N, \label{estim_diff}
    \end{align}
    where $C$ is a constant that only depends on  $T$, $ \max_{t\in [0,T]} |G(t)|$, $L'$ and $|x_0|$. On the other hand, we get from~\eqref{hat_X} that
  $$\hat{X}_t-\hat{X}_{t_{k+1}-}=\int_0^{t_{k}}[G(t-t_{\eta(s)+1})-G(t_{k+1}-t_{\eta(s)+1})][b(\xi_s)ds+\sigma(\xi_s)dW_s].$$
  Using again Jensen inequality, Itô isometry,  \blue{\eqref{SubLin} and Lemma~\ref{lem_apriori_estimate},} we get
  \begin{align*}
    \E[|\hat{X}_t-\hat{X}_{t_{k+1}-}|^2]&\le 2\int_0^{t_{k}}[G(t-t_{\eta(s)+1})-G(t_{k+1}-t_{\eta(s)+1})]^2( T\E[|b(\xi_s)|^2]+\E[|\sigma(\xi_s)|^2])ds\\
      &\le 2\omega_{G,T}^2(T/N) (1\vee T) L'\int_0^T( 1+\E[|\xi_s|^2])ds\le C\omega_{G,T}^2(T/N),
  \end{align*}
  for a constant $C$ that only depends on  $T$, $ \max_{t\in [0,T]} |G(t)|$, $L'$ and $|x_0|$. 
\end{proof}

\begin{prop}\label{prop_approx}
  Let us assume that $G:\R_+\to \R$ is a continuous kernel and that the coefficients $b,\sigma$ satisfy~\eqref{Lip}. Then, there exists a constant~$C$ depending on $T$, $ \max_{t\in [0,T]} |G(t)|$, $L$, $L'$ and $|x_0|$ such that
  $$\forall t \in[0,T],\  \E[|\hat{X}_t-X_t|^2]\le C \left(\frac{T}{N} +\omega^2_{G,T}\left( \frac T N\right)\right).$$
\end{prop} 
\begin{proof}
We get from~\eqref{SVE} and~\eqref{hat_X}:
\begin{align}
  \hat{X}_t-X_t=&
  \int_0^t G(t-s)[(b(\xi_s)-b(X_s))ds+ (\sigma(\xi_s)-\sigma(X_s))dW_s ] \notag \\
  &+\int_{0}^{t_{\eta(t)}}(G(t-t_{\eta(s)+1})-G(t-s))[b(\xi_s)ds+ \sigma(\xi_s)dW_s] \notag\\
  &- \int_{t_{\eta(t)}}^t G(t-s)[b(\xi_s)ds+ \sigma(\xi_s)dW_s].\notag
\end{align}
We deduce by using Jensen inequality, Itô isometry \blue{and~\eqref{SubLin}}
\begin{align}
  \E[|\hat{X}_t-X_t|^2]\le&6 \Bigg( \int_0^t G(t-s)^2(T+1)L^2 \E[|\xi_s-X_s|^2]ds \notag \\
  &+ \int_{0}^{t_{\eta(t)}}(G(t-t_{\eta(s)+1})-G(t-s))^2 L'(1+ \E[\xi_s^2]) ds \notag \\
  &   + \int_{t_{\eta(t)}}^t G(t-s)^2  L'(1+ \E[\xi_s^2]) ds\Bigg). \notag
\end{align}
By Lemma~\ref{lem_apriori_estimate}, there exists a constant~$C_1$ depending on $T$, $ \max_{t\in [0,T]} |G(t)|$,  $L'$ and $|x_0|$ such that 
$$\sup_{t\in [0,T]}\E[|\xi_t|^2+|\hat{X}_t|^2]\le C_1 \text{ and }\sup_{t\in [0,T]}\E[|\xi_t-\hat{X}_t|^2]\le C_1 \left(\frac{T}{N} +\omega^2_{G,T}\left( \frac T N\right)\right).$$
We get that
\begin{align*}
   \E[|\hat{X}_t-X_t|^2] &\le 6\max_{t\in [0,T]} G(t)^2L^2(T+1) \left[ 2T C_1 \left(\frac{T}{N}+2 T\omega^2_{G,T}\left( \frac T N \right)  \right) +2 \int_0^t \E[|\hat{X}_s-X_s|^2]ds \right] \\& + \omega_{G,T}^2\left( \frac T N\right) T L'(1+C_1)+\frac{T}{N} L'\max_{t\in [0,T]} G(t)^2(1+C_1).\end{align*}
Therefore, there exists a constant $C_2$ depending only on    $T$, $ \max_{t\in [0,T]} |G(t)|$, $L$,  $L'$ and $|x_0|$ such that  
$$\E[|\hat{X}_t-X_t|^2] \le C_2  \int_0^t \E[|\hat{X}_s-X_s|^2]ds +C_2\left(\frac{T}{N} +\omega_{G,T}^2(T/N)\right)$$
We conclude by using Gronwall Lemma that $\E[|\hat{X}_t-X_t|^2] \le C_2\left(\frac{T}{N} +\omega_{G,T}^2(T/N)\right)e^{C_2T}$, which gives the claim. 
\end{proof}

We conclude this paragraph by showing a comparison result between real valued SVEs.
\begin{theorem}\label{thm_comparison_result}
Let $b^1,b^2,\sigma:\R\to \R$ be Lipschitz continuous functions such that $b^1(x)\le b^2(x)$ for all $x \in \R$. Let $G:\R_+^* \to \R_+$ be continuous nonincreasing and such that there exists $\varepsilon,\eta >0$ such that
\begin{align}\label{Hyp_kernel}
&\forall T>0, \ \int_0^T G(s)^{2+\varepsilon}(u)du < \infty, \quad  \exists C_T, \forall \delta \in (0,T), \ \int_0^T (G(s+\delta)-G(s))^2 ds \le C_T \delta^\eta.
\end{align}
Let $x^1_0\le x^2_0$ and define, for $l\in \{1,2\}$,
$$X^l_t=x^l_0+\int_0^t G(t-s)b^l(X^l_s)ds+\int_0^t G(t-s)\sigma(X^l_s)dW_s, $$
where $W$ is a one-dimensional Brownian motion. 

If $G(0+)<\infty$ and $G$ preserves nonnegativity, then we have $\P(\forall t\ge 0, X^1_t\le X^2_t)=1$. The same conclusion holds when $G(0+)=+\infty$ and $G$ is completely monotone. 
\end{theorem} 
\noindent \blue{The existence and uniqueness of the processes $X^1$ and $X^2$ under the assumptions of Theorem~\ref{thm_comparison_result} is ensured by Wang~\cite{Wang}. Let us note here that the fractional kernel $G_{H}(t)=\frac{t^{H-1/2}}{\Gamma(H+1/2)}$, $H\in(0,1/2)$ satisfies the assumption~\eqref{Hyp_kernel} by Richard et al.~\cite[Example 2.1]{RTY}. }

Few results exist for the comparison of Stochastic Volterra Equations. Tudor~\cite{Tudor} obtained a comparison result for more general drift coefficients but with a diffusion term $\sigma(s,x)$ that only depend on time~$s$. His result is extended by Ferreyra and Sundar~\cite{FeSu} who consider a diffusion term in $H(t)\sigma(s,x)$, but this does not include the convolution case $G(t-s)\sigma(x)$ above. Note that our approach can be straightforwardly extended to time-dependent coefficients, i.e. for 
$$X^l_t=x^l_0+\int_0^t G(t-s)b^l(s,X^l_s)ds+\int_0^t G(t-s)\sigma(s,X^l_s)dW_s, $$
but we keep time-independent coefficients for simplicity. \blue{Finally, let us note that even for Volterra (ordinary) differential equations, Theorem~\ref{thm_comparison_result} with $\sigma\equiv 0$ appears to be new. Sato~\cite[Theorem 4]{Sato_Comparison} and later Gripenberg~\cite{Gripenberg_Comparison} state general comparison results that would amount to have in our setting $b^1(x)\le b^2(x')$ for $x\le x'$, which is more restrictive and holds if $b^1\le b^2$ and one of the two functions is nondecreasing.   } 

\begin{proof}
  From Wang~\cite[Theorems 1.1 and 1.3]{Wang}, we know that the SVE solutions~$X^1$ and $X^2$ are unique and almost surely Hölder continuous. We will prove that $\P(X^1_T\le X^2_T)=1$ for any $T>0$, which implies the claim by using the almost sure continuity.

  We first consider the case where $G$ is nonincreasing, preserves nonnegativity
 and satisfies $G(0+)<\infty$.  Let $T>0$. For $N \in \N^*$, we define the corresponding approximations $\hat{X}^1$ and $\hat{X}^2$, and denote $\xi^1$ and $\xi^2$ the processes that arise in their construction, see Equation~\eqref{def_xit}. We show by induction on $k$ that $\hat{X}^1_t\le \hat{X}^2_t$ for $t\in[0,t_k]$.
  For $k=1$, we have $\hat{X}^1_t= x^1_0\le x^2_0=\hat{X}^2_t$ for $t\in[0,t_1)$ and by using the comparison results for SDEs, see e.g. Karatzas and Shreve~\cite[Proposition 2.18 p. 293]{KS}, we get that $\hat{X}^1_{t_1}=\xi^1_{t_1-}\le  \xi^2_{t_1-} =\hat{X}^2_{t_1}$. 
  Suppose that the result is true for $k\ge 1$. Then, we have
  $$\forall l \in \{1,\dots,k\}, \ \hat{X}^2_{t_l}-\hat{X}^1_{t_l}=x^2_0-x^1_0+\sum_{\ell=1}^{l}   \frac{\hat{X}^2_{t_\ell}-\hat{X}^2_{t_\ell-} -(\hat{X}^1_{t_\ell}-\hat{X}^1_{t_\ell-})}{G(0)} G(t_l-t_\ell) \ge 0 .$$
  Since $G$ is nonincreasing and preserves nonnegativity we get by Proposition~\ref{prop_pos_gen} that 
  $\hat{X}^2_t-\hat{X}^1_t\ge 0$ for $t \in (t_k,t_{k+1})$. Then, we use again the comparison result between SDEs to get  $\hat{X}^1_{t_{k+1}}=\xi^1_{t_{k+1}-}\le  \xi^2_{t_{k+1}-} =\hat{X}^2_{t_{k+1}}$. By induction, we get for $k=N$ that $\hat{X}^1_{T}\le \hat{X}^2_{T}$, a.s. By Proposition~\ref{prop_approx},  $\hat{X}^l_{T}$ converges in $L^2$ to $X^l_T$ as $N\to \infty$, for $l\in \{1,2\}$. This gives $\P(X^1_T\le X^2_T)=1$.

  We now prove the result when $G$ is completely monotone. By the Hausdorff-Bernstein-Widder theorem~\cite[Theorem IV.12a]{Widder}, there exists a Borel measure $\mu$ on $\R_+$ such that $G(t)=\int_{(0,+\infty)} e^{-\rho t} \mu(d\rho)$. 
We consider for $K>0$ the kernel $G^K$ defined by
\begin{equation}
  G^K(t)=\int_{[0,K]} e^{-\rho t} \mu(d \rho).\label{def_G_K}
\end{equation}
We consider the Stochastic Volterra Equations associated to this kernel:
$$X^{l, K}_t=x^l_0 +\int_0^t G^K(t-s) b(X^{l,K}_s)ds+\int_0^t G^K(t-s)\sigma(X^{l,K}_s)dW_s, \ t\ge 0.$$
We note that the kernel is bounded by $G^K(0)$ and $|(G^K)'(t)|\le KG^K(0)$, so that $G^K$ satisfies~\eqref{Hyp_kernel}. It is nondecreasing and preserves nonnegativity by Theorem~\ref{thm:positivite}. We can thus apply the result of the first part of this theorem and get $\P(X^{1,K}_T \le X^{2,K}_T)=1$. On the other hand, we have $\E[|X^{l,K}_T-X^l_T|^2]\le C r(K)$ by Alfonsi and Kebaier~\cite[Proposition 3.1]{AK},
where $r(K):=\int_{(K,\infty)^2} \frac{1}{\rho_1+\rho_2} \mu(d \rho_1)\mu(d \rho_2)$. By~\eqref{Hyp_kernel}, we get
$$\infty>\int_0^T G(t)^2dt =\int_0^T \int_{\R^2}e^{-(\rho_1+\rho_2)t} \mu(d \rho_1)\mu(d \rho_2)= \int_{\R^2}\frac{1-e^{-{(\rho_1+\rho_2)}T}}{\rho_1+\rho_2} \mu(d \rho_1)\mu(d \rho_2),$$
and then $r(K)\underset{K\to \infty}\to 0$. This gives $\P(X^1_T\le X^2_T)=1$ for all $T>0$ and thus the claim.
\end{proof}

\begin{remark}
In the proof of Theorem~\ref{thm_comparison_result}, Assumption~\eqref{Hyp_kernel} is only used to get the almost sure continuity of the solutions $X^1$ and $X^2$. This allows to deduce $\P(\forall t\ge 0, X^1_t \le X^2_t)=1$ from $\forall t\ge 0, \P(X^1_t \le X^2_t)=1$. However, when $G:\R_+\to \R_+$ is continuous, we still have unique strong solutions by Berger and Mizel~\cite[Theorem 3.A]{BeMi}, and we get by the same arguments that $\forall t\ge 0, \P(X^1_t \le X^2_t)=1$.
\end{remark}

\subsection{Stochastic invariance of a convex domain: from SDEs to SVEs}

We are now in position to prove the stochastic invariance of some Stochastic Volterra Equations. Up to our knowledge, there are few works dealing about the stochastic invariance of SVEs. We mention here the articles of El Euch and al.~\cite{EFR} that gives the nonnegativity of a weak solution of the rough Cox-Ingersoll-Ross process, \blue{Abi Jaber~\cite{AJ_Bernoulli} that shows the nonnegativity of SVEs with jumps} under the assumption that the kernel is nonincreasing, with a nonincreasing resolvent of the first kind (see Definition~\ref{def_nonincr}) and Abi Jaber et al.~\cite[Theorem 3.6]{AJLP} that shows the existence of a weak solution in $\mathscr{C}=(\R_+)^d$ under suitable assumptions on $b$ and $\sigma$ and \blue{the same assumption on the kernel. Since the first version of the present paper, Abi Jaber et al.~\cite{AJCPPSF}  have shown also the existence of Volterra processes on the unit ball. The kernel assumption made in those works} is satisfied by completely monotone kernels as pointed in~\cite[Example 3.7]{AJLP}. Here, the approach is different and relies on Theorem~\ref{thm_carac_np} and Proposition~\ref{prop_pos_gen} combined with the approximating process studied in Proposition~\ref{prop_approx}. Namely, we show that the approximating process remains in this domain~$\mathscr{C}$, and this property then holds for the strong solution of the SVE. 

\begin{theorem}\label{thm_SVE_dom}
  Let $\mathscr{C}\subset \R^d$ be a nonempty closed convex subset and $x_0\in \mathscr{C}$. Let $b:\R^d \to \R^d$ and $\sigma: \R^d \to \mathcal{M}_d(\R)$ satisfy~\eqref{Lip}, and $G:\R_+^* \to  \R_+$ be continuous nonincreasing satisfying~\eqref{Hyp_kernel}.
  \begin{enumerate}
  \item Assume  that $G(0+)<\infty$, $G$  preserves nonnegativity and that the Stochastic Differential Equation
  \begin{equation}\label{SDE_xi2}\xi_t=x+\int_0^t G(0+) b(\xi_s)ds+\int_0^t G(0+)\sigma(\xi_s)dW_s, \ t\ge 0,\end{equation}
  satisfies $\P(\forall t\ge 0, \xi_t\in \mathscr{C})=1$ for all $x\in \mathscr{C}$. 
  Then, the Stochastic Volterra Equation
  \begin{equation}\label{SVE_thm}
    X_t=x_0 +\int_0^t G(t-s) b(X_s)ds+\int_0^t G(t-s)\sigma(X_s)dW_s, \ t\ge 0,\end{equation}
 satisfies  $\P(\forall t\ge 0, X_t \in \mathscr{C})=1$  for all $x_0\in \mathscr{C}$.
 \item  Assume $G(0+)=+\infty$, $G$ is 
 completely monotone and that there exists $\Lambda>0$ such that for all $\lambda> \Lambda$, the SDE 
 \begin{equation}\label{SDE_lambda}\xi_t=x+\int_0^t \lambda b(\xi_s)ds+\int_0^t \lambda \sigma(\xi_s)dW_s, \ t\ge 0,
 \end{equation}
 satisfies $\P(\forall t\ge 0, \xi_t\in \mathscr{C})=1$ for all $x\in \mathscr{C}$. Then, the SVE~\eqref{SVE_thm} satisfies  $\P(\forall t\ge 0, X_t \in \mathscr{C})=1$  for all $x_0\in \mathscr{C}$. 
  \end{enumerate}
\end{theorem}

\begin{proof}
From Wang~\cite[Theorems 1.1 and 1.3]{Wang}, we know that there exists a unique solution~$X$ to the SVE~\eqref{SVE} in both cases, and that this solution is almost surely Hölder continuous. In fact, we will only use here the continuity of the solution.   It is thus sufficient to check that $\forall t \in \R_+, \P(X_t \in \mathscr{C})=1$ : this will imply $\P(X_t \in \mathscr{C}, \forall t\in \R_+ \cap \mathbb{Q})=1$ and then $\P(X_t \in \mathscr{C}, \forall t\in \R_+)=1$ by using the almost sure continuity.

We prove the first statement. We consider $T>0$, and $\hat{X}$ the approximation defined by~\eqref{hat_X} with $N\in \mathbb{N}^*$  time steps and $G(0)=G(0+)$. We show by iteration on~$k$ that $\hat{X}_t\in \mathscr{C}$ for $t\in[0,t_k]$.  
For $k=1$, this is clear since $\hat{X}_t=x_0$ for $t\in [0,t_1)$ and then $\xi_{t_1}\in \mathscr{C}$ by using the stochastic invariance of~$\scrC$ by~\eqref{SDE_xi2}. We now suppose that $\hat{X}_t\in \mathscr{C}$ for $t\in[0,t_k]$ and show that the same result holds on~$(t_k,t_{k+1}]$. 
To do so, we write $\mathscr{C}$ as a countable intersection of half spaces: there exists families $(\alpha_\theta)_{\theta \in \Theta}$ and $(\beta_\theta)_{\theta \in \Theta}$ of vectors and real numbers with $\Theta$ countable, such that 
\begin{equation}\label{convex_halfspaces}\mathscr{C}=\cap_{\theta \in \Theta} \{ x \in \R^d, \alpha_\theta^T x +\beta_\theta \ge 0 \}.
\end{equation}
Let $\theta \in \Theta$. We have, $\alpha_\theta^T \hat{X}_t+\beta_\theta=\alpha_\theta^T x_0+\beta_\theta  +\sum_{k:t_j\le t }\frac{\alpha_\theta^T (\hat{X}_{t_{j}}-\hat{X}_{t_j-})}{G(0)}G(t-t_j)\ge 0$ for $t\in [0,t_k]$. Then, since the kernel is nonincreasing and preserves nonnegativity, we get that $\alpha_\theta^T \hat{X}_t+\beta_\theta\ge 0$ for $t\in (t_k,t_{k+1})$ by Proposition~\ref{prop_pos_gen}. Since $\theta$ is arbitrary, this shows that $\hat{X}_{t_{k+1}-}\in \mathscr{C}$. Then, we use the stochastic invariance of the SDE to get that $\hat{X}_{t_{k+1}}=\xi_{t_{k+1}-}\in \mathscr{C}$. Thus, we obtain that $\hat{X}_T \in \mathscr{C}$. Besides, by Proposition~\ref{prop_approx}, $\hat{X}_T$ converges in $L^2$ to $X_T$ as $N \to \infty$, which shows that $X_T \in \mathscr{C}$ almost surely. The claim follows.

We now prove the second point, proceeding similarly as in the proof of Theorem~\ref{thm_comparison_result}. By the Hausdorff-Bernstein-Widder theorem~\cite[Theorem IV.12a]{Widder},  $G(t)=\int_{(0,+\infty)} e^{-\rho t} \mu(d\rho)$ for a Borel measure $\mu$ on $\R_+$, and we define for $K>0$ the kernel $G_K$ by~\eqref{def_G_K}. We take  $K$ large enough so that $G_K(0)>\Lambda$. We consider the Stochastic Volterra Equation associated to this kernel:
$$X^{K}_t=x_0 +\int_0^t G^K(t-s) b(X^{K}_s)ds+\int_0^t G^K(t-s)\sigma(X^{K}_s)dW_s, \ t\ge 0.$$
The kernel $G^K$  satisfies~\eqref{Hyp_kernel} (see the proof of Theorem~\ref{thm_comparison_result}) and preserves nonnegativity by Theorem~\ref{thm:positivite}. We can thus apply the result of the first part of this Theorem and get $\P(X^K_T \in \mathscr{C})=1$ for any $T>0$. Since $\E[|X^K_T-X_T|^2]\le C r(K)$ by~\cite[Proposition 3.1]{AK} with $r(K)\underset{K\to \infty}\to 0$, we get  $\P(X_T\in \mathscr{C})=1$ and thus the claim.
\end{proof}

\begin{remark}
  We see from the proof that the conclusions of Theorem~\ref{thm_SVE_dom}~(2) and  Theorem~\ref{thm_comparison_result}) also holds  when $G:\R_+^*\to \R_+$ is such that there exists a family of continuous, nonincreasing and nonnegativity preserving kernels $(G^K)_{K\in \R_+^*}$ such that $G^K(0)\underset{K\to \infty}{\to} \infty$, $\int_0^T(G(t)-G^K(t))^2dt\underset{K\to \infty}{\to}  0$ for all $T>0$, and $G^K$ satisfies~\eqref{Hyp_kernel} for all $K>0$. 
  When $G$ is completely monotone, the family~\eqref{def_G_K} satisfy all these requirements, as well as the family $G_K(t)=G(t+1/K)$.
  \blue{It would be interesting to know if more general kernels~$G$ could fit these conditions (typically one would expect that the resolvent of $G$ exists and satisfies $R(ds)=r(s)ds$ with  $r:\R_+^* \to \R_+$ locally integrable and nonincreasing, in view of Theorem~\ref{thm_char_pos}) and, for them    to build such  an approximating family $(G^K)_{K\in \R_+^*}$. This is left for further research. }
\end{remark}

\begin{remark}\label{rk_Wishart}\blue{We have made in the present section standard Lipschitz assumptions on the coefficients.  This choice allows to present comparison and stochastic invariance results for strong solutions. Roughly speaking, the main message is that for nonnegativity preserving kernels, these properties hold if they hold for the corresponding SDE (SDEs when $G(0+)=+\infty$). We may then wonder if this can be transposed to weak solution for the stochastic invariance.  From
the proof of Theorem~\ref{thm_SVE_dom}, we see indeed that the approximating process $\hat{X}$ remains in $\mathscr{C}$ even if $b$ and $\sigma$ are not Lipschitz continuous, provided that a weak solution to~\eqref{SDE_xi2} in $\mathscr{C}$ exists. We could thus relax the Lipschitz assumption, study the tightness and the possible weak convergence of (a subsequence) of $\hat{X}$ toward a weak solution of the SVE~\eqref{SVE_thm}. This is left for further research.} This approach may be relevant in particular for Affine Volterra processes. For example, one may justify like this the existence of a Volterra Wishart process
\begin{equation}\label{Volterra_Wishart} X_t=X_0+\int_0^t G(t-s) [\alpha a^\top a +bX_s +X_s b^\top ]ds + \int_0^t G(t-s) [\sqrt{X_s} dB_s a +a^\top dB_s^\top \sqrt{X_s}],
\end{equation}
with $\alpha \in \R$ $a,b\in \mathcal{M}_d(\R)$ and $B$ a Brownian motion in  $\mathcal{M}_d(\R)$, under the condition\footnote{Let us recall that the classical Wishart SDE $X_t=X_0+\int_0^t  [\alpha a^\top a +bX_s +X_s b^\top ]ds + \int_0^t  [\sqrt{X_s} dB_s a +a^\top dB_s^\top \sqrt{X_s}]$  has a unique weak solution for $\alpha \ge d-1$.} that $\frac{\alpha}{G(0+)}\ge d-1$, for a kernel $G$ that preserves nonnegativity and such that $G(0+)<\infty$. On the other hand, it does not seem possible to define~\eqref{Volterra_Wishart} when $G(0+)=\infty$ and $d\ge 2$, since the condition $\frac{\alpha}{\lambda}\ge d-1$ cannot be satisfied for large values of~$\lambda$. This somehow justifies the approaches of Cuchiero and Teichmann~\cite{CuTe} and Abi Jaber~\cite{AJ_Wishart} that consider square of Volterra Matrix Gaussian processes rather than~\eqref{Volterra_Wishart} to extend classical Wishart processes.      
\end{remark}
\begin{example}\label{ex_VGBM}
  Let $G$ be a kernel satisfying the assumptions of Theorem~\ref{thm_SVE_dom}, $\mu \in \R$ and $\sigma\ge 0$. By Theorem~\ref{thm_SVE_dom}, for any $x_0\ge 0$, the solution of
  $$X_t=x_0+\int_0^t G(t-s) \mu X_s ds+\int_0^t G(t-s) \sigma X_s dW_s,$$
remains nonnegative.  This process is a natural extension to Volterra type dynamics of the geometric Brownian motion.    
\end{example}
\blue{
\begin{example}\label{ex_simplex}(A stochastic Volterra equation on the simplex.) Let $d\ge 2$, $\mathscr{C}=\{x \in \R^d: x^k\ge 0, \ 1\le k \le d \text{ and } \sum_{k=1}^d x^k=1\}$ be the closed simplex. Let $\pi \in \mathcal{C}$, $a\le 0$, $b(x)=a(x-\pi)$, $h:\R \to \R$ be a Lipschitz function such that $h(0)=0$ and $\sigma: \R^d \to    \mathcal{M}_d(\R)$ defined by 
  $$\sigma_{k,k}(x)=(1-x^k) h(x^k), \  \sigma_{k,l}(x)=-x^k h(x^l) \text{ for }l \not = k.$$
  Note that $\sigma$ is Lipschitz on $[0,1]^d$. By Theorem~\ref{thm_SVE_dom}, if the kernel $G$ satisfies the condition therein, we get that for any $x_0\in \mathscr{C}$, the solution of
  $$X_t=x_0+\int_0^t G(t-s) b(X_s) ds+\int_0^t G(t-s) \sigma(X_s) dW_s,$$
remains in the simplex~$\mathscr{C}$. For this, we have to check that for any $\lambda>0$, the set $\mathscr{C}$ is invariant by the SDE
$$\xi^k_t=x^k+\int_0^t \lambda a (\xi^k_s- \pi^k) ds + \int_0^t \lambda  \left( h(\xi^k_s) dW^k_s- \xi^k_s \sum_{l=1}^d h(\xi^l_s) dW^l_s \right), \ 1\le k\le d. $$
This is the case, as shown by Gourieroux and Jasiak~\cite{GoJa} who introduced this SDE with a square-root function~$h$. Indeed, if for example $\xi^k$ vanishes, the diffusion term of $\xi^k$ vanishes as well while the drift is equal to $-\lambda a\pi^k \ge 0$. The coordinates stay therefore nonnegative. Besides we have $$d\left( \sum_{k=1}^d \xi^k_t\right)=\lambda a\left(\sum_{k=1}^d \xi^k_t-1\right)dt+\lambda\left(1-\sum_{k=1}^d \xi^k_t\right) \left( \sum_{l=1}^d h(\xi^l_t)dW^l_t \right) ,$$ so that $\sum_{k=1}^d \xi^k_t=1$ for all $t\ge 0$ if $x\in\mathscr{C}$.     
\end{example}
}
Theorem~\ref{thm_SVE_dom} relies on the stochastic invariance of Stochastic Differential Equation. This topic has been well studied in the literature. Da Prato and Frankowska~\cite{DPFr} have precisely studied the case of closed convex domain, and we refer to Abi Jaber et al.~\cite{AJBI} for a general  necessary and sufficient condition of stochastic invariance, as well as an up to date review on the literature. 

A natural question however is to determine if we may find some coefficients $b$ and $\sigma$ such that the SVE~\eqref{SVE_thm} satisfies $\P(\forall t \ge 0, X_t \in \mathscr{C})$ for any kernel $G$ satisfying the conditions of Theorem~\ref{thm_SVE_dom}.  This amounts to determine $b$ and $\sigma$ such that stochastic invariance holds for the SDE~\eqref{def_xit}, for any $\lambda>0$. 
In the case $d=1$ and $\mathscr{C}=\R_+$, this is true if $b(0)\ge 0$ and $\sigma(0)=0$, see e.g. Revuz and Yor~\cite[Theorem IX.3.7]{ReYo}. 
More generally, let us assume that~\eqref{Lip} still holds and that 
\begin{equation}\label{hyp_C11loc}
\R^d \ni x\mapsto \sigma \sigma^\top(x)=:C(x) \text{ is a } \mathcal{C}^1 \text{ function with locally Lipschitz derivative.}
\end{equation}
Then, \cite[Theorem 2.3]{AJBI} gives that the solution of~\eqref{def_xit} satisfies $\P(\xi_t \in \mathscr{C}, t\ge 0)$ if, and only if,
\begin{equation}
  \forall x \in \mathscr{C},u\in \mathcal{N}^1_{\mathscr{C}}(x), \  
  \begin{cases}
    C(x)u=0 \\
    \langle u, \lambda b(x) -\frac {\lambda^2} 2 \sum_{j=1}^d DC^j(x)(CC^+)^j(x)  \rangle \le 0.
  \end{cases}
\end{equation} 
Here, we take back the notation of~\cite{AJBI}: $\mathcal{N}^1_{\mathscr{C}}(x)=\{u \in \R^d: \langle u,y-x \rangle  \le o(|y-x|), y\in \mathscr{C} \}$ is the first order normal cone at $x$ and $(CC^+)^j(x)$ is the $j$th column of  $(CC^+)(x)$, where $C^+$ is the Moore-Penrose pseudoinverse of~$C$. Thus, we have $\P(\xi_t \in \mathscr{C}, t\ge 0)$ for any $\lambda>0$ if, and only if,
\begin{equation}\label{cond_all_lambda}
  \forall x \in \mathscr{C},u\in \mathcal{N}^1_{\mathscr{C}}(x), \  
  \begin{cases}
    C(x)u=0 \\
    \langle u,  b(x)  \rangle \le 0, \   \langle u, \sum_{j=1}^d DC^j(x)(CC^+)^j(x)  \rangle \ge 0.
  \end{cases}
\end{equation} 
Let us note that the convexity of $\mathscr{C}$ gives by~\cite[Remark 2.9]{DPFr} and~\cite[Proposition 2.4]{AJBI}, under some smoothness assumptions, that $\langle u, \sum_{j=1}^d DC^j(x)(CC^+)^j(x)  \rangle\le 0$ for $x$ being on the boundary of $\mathscr{C}$ and therefore the last condition of~\eqref{cond_all_lambda} can simply be rewritten in this case as $\langle u, \sum_{j=1}^d DC^j(x)(CC^+)^j(x)  \rangle = 0$. 

\begin{example}
  Let $\mathscr{C}=\R_+ \times \R^{d-1}$, $b, \sigma$ satisfy~\eqref{Lip} and~\eqref{hyp_C11loc}. Assume that $b_1(x)\ge 0$ and $\sigma_{1k}(x)=0$ for $k\in \{1,\dots,d \}$ when $x_1=0$. We note $(e_l)_{1\le l \le d}$ the canonical basis of $\R^d$. Then, for $x\in \R^d$ such that $x_1=0$, we have $\mathcal{N}^1_{\mathscr{C}}(x)=\{\lambda e_1, \lambda<0 \}$, $C_{1,j}(x)=C_{j,1}(x)=0$ for $j\in \{1,\dots, d\}$. This gives  $(CC^+)^1(x)=0$ and $(CC^+)^j(x) \in \textup{Span}(e_2,\dots,e_d)$ for $j\ge 2$.
  Since  $\partial_l C_{1,j}(x)=\lim_{\varepsilon\to 0} \frac{ C_{1,j}(x+\varepsilon e_l)-C_{1,j}(x)}{\varepsilon}=0$ for $l\in \{2,\dots,d\}$, we get $\langle e_1, DC^j(x) e_l\rangle=\partial_lC_{1,j}(x)=0$ and thus $\langle e_1, DC^j(x)(CC^+)^j(x)  \rangle=0$. Therefore,~\eqref{cond_all_lambda} is satisfied. By Theorem~\ref{thm_SVE_dom}, if the kernel $G$ satisfies the conditions therein, the SVE~\eqref{SVE_thm} satisfies $\P(\forall t\ge 0, X_t \in \mathscr{C})=1$. 
\end{example}

\section{Second order approximation schemes for SVEs with discrete completely monotone kernels}\label{Sec_approx}
\blue{In Section~\ref{Sec_SVEs}, we have obtained comparison and stochastic invariance results for Stochastic Volterra Equation by using an approximation scheme. This approximation scheme converges strongly at a rate given by Proposition~\ref{prop_approx}, and the proof principle was to propagate the properties of the approximation scheme to the limit process. The approximation scheme used in Section~\ref{Sec_SVEs} is obtained by splitting, and it is well known that this method is interesting to get high order of convergence for the weak error. The goal of this section is to develop a scheme that achieves a weak rate of convergence of order~2 and stays in the same closed convex domain as the underlying SVE, see Theorem~\ref{thm_second_order_SVE} below.  To get this result, we consider multi-exponential convolution kernels in order to have an SDE representation of the SVE and to use existing regularity results for the Cauchy problem. The invariance argument relies, as in Section~\ref{Sec_SVEs}, on Proposition~\ref{prop_pos_gen}. This approach is then applied to the case of multifactor Cox-Ingersoll-Ross and Heston model: the obtained approximation scheme shows on our numerical experiments  a convergence of order~2 and outperforms the corresponding Euler scheme.        }

\subsection{Second order schemes for SDEs}

Let us start by recalling some results on the analysis of the weak error for approximation schemes of SDEs. Let us define
\begin{align*}
  \Cpol{\R^d}:=\{f: &\R^d\to \R, \ f\in\mathcal{C}^\infty \ \text{s.t.} \\ &\forall \alpha \in \N^d, \exists C_\alpha>0, e_\alpha \in \N^*, \forall x \in \R^d, |\partial_\alpha f(x)|\le C_\alpha(1+|x|^{e_\alpha}) \},
\end{align*}
the set of real valued $\mathcal{C}^\infty$ functions with all derivatives of polynomial growth. Here, we use the notation $\partial_\alpha=\partial_1^{\alpha_1}\dots\partial_d^{\alpha_d}$ for $\alpha \in \N^d$. For $f\in \Cpol{\R^d}$, we say that $(C_\alpha,e_\alpha)_{\alpha \in \N^d}$ is a good sequence for $f$ if we have $|\partial_\alpha f(x)|\le C_\alpha(1+|x|^{e_\alpha})$ for all $x\in \R^d$, and all $\alpha \in \N^d$. 

We consider coefficients $b:\R^d\to \R^d$ and $\sigma:\R^d\to \mathcal{M}_d(\R)$ with sublinear growth (i.e. $\exists C\in \R_+,\forall x,\ |b(x)|+|\sigma(x)|\le C(1+|x|)$) such that for any $1\le k,l \le d$, $b_k,(\sigma \sigma^\top)_{k,l} \in \Cpol{\R^d}$. We introduce $$L f(x)=\sum_{k=1}^d b_k(x) \partial_kf(x)+\frac 12 \sum_{k,l=1}^d (\sigma \sigma^\top (x))_{k,l} \partial_k\partial_l f(x),$$ the infinitesimal generator of the Stochastic Differential Equation 
$$ \xi^x_t= x+\int_0^t b(\xi_s^x)ds +\int_0^t \sigma (\xi_s^x)dW_s, \ t\ge 0, x\in \R^d,$$
where $W$ is a $d$-dimensional Brownian motion. We assume that this SDE has a unique weak solution. We consider a discretization scheme $\varphi(x,t,Z)$, where  $\varphi:\R^d \times \R_+ \times \R^{d_Z} \to \R^d$ and $Z$ is a random variable taking values in $\R^{d_Z}$, with $d_Z\in \N^*$. Let $T>0$, $N\in \N^*$ and $(Z_\ell)_{\ell\ge 1}$ be an i.i.d. sequence of random variable having the same law as~$Z$. We then define the scheme $\hat{\xi}^N$ on the regular time grid $\ell \frac{T}{N}$ with $0\le \ell \le N$ by:
\begin{equation}\label{scheme_xiN} \hat{\xi}^N_0=x, \   \hat{\xi}^N_{\ell \frac TN}=\varphi\left(\hat{\xi}^N_{(\ell-1)\frac TN},\frac T N, Z_\ell \right), \ \ell \in \{1,\dots,N\}.
\end{equation}

\begin{theorem}\label{Thm_TT}
   Let $\nu \in \N^*$. Under the above framework, let us assume that:
  \begin{enumerate}
    \item $f \in \Cpol{\R^d}$ is such that the function $u(t,x)= \E[f(\xi^x_{T-t})]$ is $\mathcal{C}^\infty$ on $[0,T] \times \R^d$, solves $\partial_t u(t,x)=-L u(t,x)$, and satisfies
    \begin{equation}\label{cont_derivees}
    \forall l \in \N, \alpha \in \N^d,\exists C_{l,\alpha},e_{l,\alpha}>0, \forall x \in
    \R^d, t\in[0,T],\   |\partial_t^l \partial_\alpha u (t,x)| \le C_{l,\alpha}(1+\|x\|^{e_{l,\alpha}}).
    \end{equation}
    \item \begin{enumerate} \item $\forall q \in \N^*,\exists N_q \in \N^*, \sup_{ N\ge N_q, 0\le \ell \le N } \E[|\hat{\xi}^N_{\ell \frac TN}|^q]<\infty$,
      \item for any function $f\in \Cpol{\R^d}$ with a good sequence  $(C_\alpha,e_\alpha)_{\alpha \in \N^d}$, there exist $C,E,\eta>0$ depending only on $(C_\alpha,e_\alpha)_{\alpha \in \N^d}$ such that
    \begin{equation}\label{potential_nu}\forall t \in (0,\eta), \forall x \in \R^d, \left|\E[f(\varphi(x,t,Z))]-\sum_{k=0}^\nu \frac{t^k}{k!} L^kf(x) \right|\le C t^{\nu+1}(1+|x|^E).\end{equation} 
    \end{enumerate}
  \end{enumerate}
Then, there is $K>0$ and  $\bar{N} \in \N$, such that 
$\forall N \ge \bar{N}, \  |\E [f(\hat{\xi}^N_T)]-\E [f(\xi^x_T)]|\le \frac{K}{N^\nu}.$    
\end{theorem}
\noindent This result has been shown by Talay and Tubaro~\cite{TaTu} for the Euler-Maruyama scheme (with $\nu=1$) and can be found in the present form in Alfonsi~\cite[Theorem 2.3.8]{AA_book}. Note that when $b$ and $\sigma$ have bounded derivatives of any order, \eqref{cont_derivees} is satisfied by~\cite[Lemma 2]{TaTu} for any $f\in \Cpol{\R^d}$.

We see from Theorem~\ref{Thm_TT} that the key property for a scheme to lead to a weak error of order~$\nu$ is to have~\eqref{potential_nu}. We say that the scheme $\varphi(x,t,Z)$ is a {\it potential $\nu$-th order scheme for~$L$} if~\eqref{potential_nu} holds.
An important result is that potential second order schemes can be constructed inductively by splitting the infinitesimal generator and by using scheme composition, see~\cite[Corollary 2.3.14]{AA_book}. This idea has been used by Ninomiya and Victoir~\cite{NV} who have proposed a general way to obtain second order schemes. Namely, they consider the following splitting
$$ L=V_0+ \frac 12 \sum_{k=1}^d V_k^2,$$
with $V_k=\sum_{i=1}^d \sigma_{i,k}\partial_i$ for $1\le k\le d$  and $V_0=\sum_{i=1}^d\left[b_i-\frac 12 \sum_{k=1}^d \sum_{j=1}^d \partial_j \sigma_{i,k}\sigma_{j,k} \right]\partial_i$.
These operators can be written $V_k f(x)=v_k(x).\nabla f(x)$ with $v_k\in \Cpol{\R^d}$. We assume that $|v_k(x)|\le C(1+|x|)$, so that the Ordinary Differential Equations
$$ \frac{d}{dt} X_k(t,x)=v_k(X_k(t,x)), t \in \R, \  X_k(0,x)=x\in\R^d,$$ 
are well defined.  The Ninomiya-Victoir scheme is then defined as follows:
\begin{equation} \label{def_NV}\varphi(x,t,Z)=\begin{cases} X_0(t/2,X_1(\sqrt{t}Z^1,\dots X_d(\sqrt{t}Z^d,X_0(t/2,x) ))) \text{ if }Z^{0}>0\\  X_0(t/2,X_d(\sqrt{t}Z^d,\dots X_1(\sqrt{t}Z^1,X_0(t/2,x) ))) \text{ otherwise,}\end{cases}\end{equation}
where $Z=(Z^0,Z^1,\dots,Z^d)$ are i.i.d. standard normal variables. 
It satisfies \eqref{potential_nu} with $\nu=2$ under the above assumptions, see~\cite[Theorem 2.3.17]{AA_book}.

\subsection{Second order schemes for multifactor SVEs}

We now apply the splitting technique to Stochastic Volterra Equations with a discrete completely monotone kernel: 
\begin{equation}\label{discrete_cm}
  G(t)=\sum_{i=1}^n \gamma_i e^{-\rho_i t}.
\end{equation} 
We call "multifactor SVEs" these equations. Here, without loss of generality, we assume that $\gamma_i>0$ for all $i\in \{1,\dots,n\}$ and $0\le \rho_1<\dots <\rho_n$. 
We consider the following Stochastic Volterra Equation
\begin{equation}\label{SVE_scheme}
  X_t =x_0+ \int_0^t G(t-s)b(X_s)ds +  \int_0^t G(t-s)\sigma (X_s)dW_s.
\end{equation}
It is well known (see e.g. Alfonsi and Kebaier~\cite[Proposition 2.1]{AK}) that the solution of the SVE~\eqref{SVE_scheme} is given by $X_t=x_0+\sum_{i=1}^n \gamma_i  X^i_t$, where ${\bf X}:=(X^1,\dots,X^n)$ solves the Stochastic Differential Equation in $(\R^d)^n$:
\begin{align} \label{SDE_dcm}
  X^i_t=-\rho_i \int_0^t X^i_s ds + \int_0^t b(X_s)ds + \int_0^t \sigma(X_s)dW_s, \ i=1,\dots,n,\  t\ge 0. 
\end{align}
Let us note that if $b:\R^d\to \R^d$ and $\sigma:\R^d\to \mathcal{M}_d(\R)$ have sublinear growth and are such that $b_k,(\sigma \sigma^\top)_{k,l}\in \Cpol{\R^d}$ for $1\le k,l\le d$, the same holds for the functions $(\R^d)^n\ni(x^1,\dots,x^n)\mapsto -\rho_i x^i + b(x_0 +\sum_{i=1}^n \gamma_i x^i)$ and $(x^1,\dots,x^n)\mapsto  \sigma(x_0 +\sum_{i=1}^n \gamma_i x^i)$, $i\in\{1,\dots,n\}$, so that the SDE~\eqref{SDE_dcm} still falls into the framework of Theorem~\ref{Thm_TT}. We note also that these functions have bounded derivatives if $b$ and $\sigma$ have bounded derivatives, in which case we have the estimates~\eqref{cont_derivees}  for the SDE~\eqref{SDE_dcm} by Talay and Tubaro~\cite[Lemma 2]{TaTu}.

The infinitesimal generator of~\eqref{SDE_dcm} is given by 
\begin{equation}\label{infgen_SDEext}
  \mathcal{L}f({\bf x})=\sum_{i=0}^{n-1}\sum_{k=1}^d (b_k(x)-\rho_{i+1}) \partial_{k+i\times d}f({\bf x})+ \frac 12 \sum_{i,j=0}^{n-1}\sum_{k,l=1}^d (\sigma\sigma^\top(x))_{k,l} \partial_{k+i\times d} \partial_{l+j\times d}f({\bf x}),
\end{equation}
with ${\bf x}=(x^1,\dots,x^n)$, $x=x_0+\sum_{i=1}^n \gamma_i x^i$ and $f:(\R^d)^n \to \R$ is a twice continuously differentiable function. We use the splitting $\mathcal{L}=\mathcal{L}_1+\mathcal{L}_2$, where $\mathcal{L}_1=- \sum_{i=0}^{n-1}\sum_{k=1}^d \rho_{i+1} \partial_{k+i\times d}f$ is the infinitesimal generator associated to
\begin{align} \label{SDE_dcm_1}
  dX^i_t=-\rho_i  X^i_t dt , i=1,\dots,n,\  t\ge 0,
\end{align}
and $\mathcal{L}_2$ is the infinitesimal generator of the following SDE 
\begin{align} \label{SDE_dcm_2}
  dX^i_t=b(X_t)dt + \sigma(X_t)dW_t \text{ with } X_t=x_0+\sum_{i=1}^n\gamma_iX^i_t,\ \  i=1,\dots,n,\  t\ge 0. 
\end{align}
Equation~\eqref{SDE_dcm_1} is a linear ODE that can be solved exactly, and we note 
\begin{equation}\label{def_psi1}
\psi_1((x^1,\dots,x^n),t)=(x^1e^{-\rho_1 t},\dots,x^n e^{-\rho_n t})  
\end{equation}
the solution. By~\cite[Remark 2.3.7]{AA_book}, it thus satisfies~\eqref{potential_nu} for any $\nu\ge 1$. 
On the other hand, the solution of~\eqref{SDE_dcm_2} can be expressed explicitly from the solution of the SDE with coefficients $G(0)b$ and $G(0)\sigma$. Namely, let $\xi^x$ be the solution of 
 \begin{equation}\label{SDE_xi} \xi^x_t=x+\int_0^t G(0)b(\xi^x_s)ds + \int_0^t G(0)\sigma(\xi^x_s)dW_s,
 \end{equation}
 with $x=x_0+\sum_{i=1}^n \gamma_i x^i$. Then, we check easily that $X^i_t=x^i+\frac{\xi^x_t-x}{G(0)}$, $i\in \{1,\dots,n\}$, is the solution of~\eqref{SDE_dcm_2} starting from $(x^1,\dots,x^n)$ and that $X_t=\xi^x_t$.
 \begin{remark}
  The approximation presented in Subsection~\ref{Subsec_SVEapprox} precisely consists in integrating on each time-step~\eqref{SDE_dcm_1} and then~\eqref{SDE_dcm_2} when $G(t)=\sum_{i=1}^n\gamma_i e^{-\rho_i t}$. Note that the exact integration of~\eqref{SDE_dcm_2} can seldom be implemented in practice. We have proved and used in Section~\ref{Sec_SVEs} strong convergence results, but this scheme leads also in principle to a weak approximation of order one by~\cite[Proposition 2.3.12]{AA_book}. Here, we develop a weak approximation of order two that can be implemented for computational purposes. \blue{It consists in integrating on a half time step~\eqref{SDE_dcm_1}, then on a time step~\eqref{SDE_dcm_2} and again on a half time step~\eqref{SDE_dcm_1}.}
 \end{remark}

 Let us define, for $\bfx=(x^1,\dots,x^n) \in \R^n$ and $y\in \R^d$, 
 \begin{equation}\label{def_Axy}
  A_{\bfx}(y)=\left(x^1+\frac{y-x}{G(0)},\dots,x^n+\frac{y-x}{G(0)} \right),
 \end{equation}
 where $x=x_0+\sum_{i=1}^d \gamma_i x^i$.  The next proposition enables us to get a second order scheme for~$\cL_2$.

\begin{lemma}\label{lem_NV}
The Ninomiya and Victoir scheme of the SDE~\eqref{SDE_dcm_2} is given by 
\begin{equation}\label{def_psi2}
  \psi_2(\bfx,t,Z)= A_{\bfx}(\varphi(x,t,Z)),
\end{equation}
where $x=x_0+\sum_{i=1}^d \gamma_i x^i$ and $\varphi(x,t,Z)$ is the Ninomiya and Victoir scheme of the SDE~\eqref{SDE_xi}. 
\end{lemma}

The proof of this Lemma is postponed to Appendix~\ref{App_proof}. We now state the main result of this subsection. For simplicity, we reinforce the assumptions on the coefficients $b$ and $\sigma$ to get automatically the estimates~\eqref{cont_derivees}. 

\begin{theorem}\label{thm_second_order_SVE}
  Let $b:\R^d\to \R^d$ and $\sigma:\R^d\to \mathcal{M}_d(\R)$ be $\mathcal{C}^\infty$ with bounded derivatives, $G(t)=\sum_{i=1}^n \gamma_i e^{-\rho_i t}$ with $0\le \rho_1<\dots<\rho_n$ and $\gamma_i>0$ for $i=1,\dots,n$. 
Let us define the following scheme 
$$\Psi(\bfx ,t,Z):=\psi_1(\psi_2(\psi_1( \bfx,t/2),t,Z),t/2),$$
where $\psi_1$ and $\psi_2$ are respectively given by~\eqref{def_psi1} and~\eqref{def_psi2}. 

Let us define $\hat{\bfX}^N_{0}=0$, $\hat{\bfX}^N_{\ell\frac TN}=\Psi\left(\hat{\bfX}^N_{(\ell-1)\frac TN } ,\frac TN,Z_\ell \right)$ for $\ell\in \{1,\dots,N\}$ and $\hat{X}^N_{\ell \frac TN}=x_0+ \sum_{j=1}^n \gamma_j \hat{X}^{N,j}_{\ell \frac TN }$. Then, we have 
$$\forall f \in \Cpol{\R^d}, \exists K>0, \bar{N} \in \N^*, \ \forall N\ge \bar{N}, |\E[f(\hat{X}^N_T)]-\E[f(X_T)]|\le \frac{K}{N^2}. $$
Besides, if $\mathscr{C}$ is a nonempty closed convex set such that the Ninomiya and Victoir scheme for~\eqref{SDE_xi} used in $\psi_2$ satisfies $\varphi(x,t,Z)\in \mathscr{C}$ a.s. for all $x\in \mathscr{C}$, and if $x_0\in \mathscr{C}$, then we have  $\hat{X}^N_{\ell \frac TN}\in \mathscr{C}$ a.s for all $0\le \ell \le N$.
\end{theorem}
\begin{proof}
By Lemma~\ref{lem_NV}, $\psi_2(\bfx,t,Z)$ is the Ninomiya and Victoir scheme for the SDE~\eqref{SDE_dcm_2}, while $\psi_1(\bfx,t)$ is the exact solution of the ODE~\eqref{SDE_dcm_1}. By using~\cite[Theorem 2.3.17 and Corollary 2.3.14]{AA_book}, we get that~\eqref{potential_nu} is satisfied by $\Psi$ for $\cL$ and $\nu=2$, i.e for any $f\in \Cpol{(\R^d)^n}$, there exist $C,E,\eta>0$ that only depend on a good sequence of~$f$ such that
$$\forall t \in (0,\eta), \ \left|\E\left[f(\Psi(\bfx ,t,Z))\right]-f(\bfx)-t \cL f (\bfx) -\frac {t^2}2 \cL^2 f(\bfx) \right|\le C t^3(1+|\bfx|^E).$$ Besides, since $b$ and $\sigma$ have bounded derivatives, \eqref{cont_derivees} is satisfied by~\cite[Lemma 2]{TaTu} and the scheme has uniformly bounded moments~\cite[Corollary 2.3.18]{AA_book}. This gives the first claim by Theorem~\ref{Thm_TT}.

For the second claim, we observe that 
$$\hat{X}^N_{\ell \frac TN}=x_0+\sum_{l=1}^\ell \delta_l G\left((\ell-l+\frac 12)\frac TN \right), $$
with $\delta_l=\varphi(\frac TN,\hat{X}^N_{(l-1/2)\frac TN},Z_l)-\hat{X}^N_{(l-1/2)\frac TN}$, where 
$$\hat{X}^N_{(\ell-1/2)\frac TN}=x_0+\sum_{l=1}^{\ell-1} \delta_l G\left((\ell-l)\frac TN \right).$$
We prove by induction on $\ell$ that $\hat{X}^N_{(\ell-1/2)\frac TN} \in \scrC$. This is true for $\ell=1$ since $x_0 \in \scrC$. Suppose now that $\hat{X}^N_{(\ell-1/2)\frac TN} \in \scrC$. Then, we know by assumption that $\varphi(\frac TN,\hat{X}^N_{(\ell-1/2)\frac TN},Z_\ell) \in \scrC$. Therefore, $x_0+\sum_{l=1}^{\ell} \delta_l G\left((\ell-l)\frac TN \right)=\varphi(\frac TN,\hat{X}^N_{(\ell-1/2)\frac TN},Z_\ell) \in \scrC$. Then, writing $\scrC$ as the intersection of half-spaces (see~\eqref{convex_halfspaces}) and using Proposition~\ref{prop_pos_gen}, we get that $\hat{X}^N_{(\ell+1-1/2)\frac TN} \in \scrC$, and also $\hat{X}^N_{\ell\frac TN} \in \scrC$.
\end{proof}

We have stated Theorem~\ref{thm_second_order_SVE} in a quite general framework. It applies to SVEs with coefficients $b$ and $\sigma$ that are well defined on the whole space~$\R^d$. It also applies to SVEs defined on a closed convex set~$\scrC\subset \R^d$, for which it is possible to extend the coefficients $b:\scrC\to \R^d$ and $\sigma:\scrC \to \mathcal{M}_d(\R)$ to smooth functions on~$\R^d$. This is the case, for example of the "Volterra geometric Brownian motion" (Example~\ref{ex_VGBM}). However, in some other cases such as the square-root diffusion coefficient, this is not possible to do such an extension. The main technical difficulty is then to determine the domain associated to the SDE~\eqref{SDE_dcm}  
$$\mathscr{D}=\{\bfx \in (\R^d)^n: \forall t\ge 0, x_0+\sum_{i=1}^n \gamma_i X^i_t \in \scrC \},$$  
where $dX^i_t=-\rho_iX^i_tdt +b(x_0+\sum_{j=1}^n\gamma_j X^j_t)dt + \sigma(x_0+\sum_{j=1}^n\gamma_j X^j_t)dW_t,\  X^i_0=x^i$ and then get the estimates~\eqref{cont_derivees} for $\bfx \in \mathscr{D}$. This is left for further research, and we just show in the next subsection how to get the estimate~\eqref{potential_nu} with $\nu=2$ for the multifactor Cox-Ingersoll-Ross process. Let us note that  $\mathscr{D}$ is nonempty since $0\in \mathscr{D}$,  and thus the support of the distribution of $(X^1_t,\dots,X^n_t)$ given by~\eqref{SDE_dcm} is included in $\mathscr{D}$ for any $t\ge 0$.

\begin{remark}
  It is interesting to notice that the SVE~\eqref{SVE_scheme} may be well defined on the closed convex set $\mathscr{C}$ while the closed convex set $\{ (x^1,\dots,x^n) \in (\R^d)^n: x_0 +\sum_{i=1}^n \gamma_i x^i \in \mathscr{C} \}$ is not stochastically invariant by the SDE~\eqref{SDE_dcm}. To illustrate this, let us take Example~\ref{ex_VGBM} with $x_0>0$, $\mu=0$, $\sigma>0$ and $G(t)=\gamma_1 e^{-\rho_1 t}+\gamma_2 e^{-\rho_2 t}$, with $\gamma_1,\gamma_2>0$ and $0<\rho_1<\rho_2$. Let $x^1,x^2\in \R$ such that $x_0+\gamma_1x^1+\gamma_2x^2\ge 0$, and
  \begin{equation}\label{volterra_geom_2}
  X^i_t=x^i - \rho_i \int_0^t X^i_s ds +  \int_0^t \sigma \times (x_0+\gamma^1X^1_s+\gamma^2X^2_s)dW_s,\ i\in \{1,2\}.
  \end{equation}
  We have by straightforward calculations $\frac{d}{dt}\big|_{t=0}\E[x_0+\gamma^1X^1_t+\gamma^2X^2_t]=-\rho_1\gamma^1x^1-\rho_2\gamma^2x^2$. For $\gamma_2x^2=-(x_0+\gamma_1x^1)$, this derivative is equal to $(\rho_2-\rho_1)\gamma^1x^1+\rho_2 x_0$ and gets negative when $x^1\to -\infty$, which shows that  $\{ (x^1,x^2) \in (\R^d)^2: x_0 +\sum_{i=1}^2 \gamma_i x^i \ge 0 \}$ is not stochastically invariant by the SDE~\eqref{volterra_geom_2}. \blue{More generally, the question of the stochastic invariance of the Markovian lift in a quite general setting has been addressed by Abi Jaber and El Euch \cite[Theorem 4.1]{AJEE2}.}
\end{remark}

\subsection{Second order schemes for the multifactor Cox-Ingersoll-Ross and Heston models}

We still consider the kernel $G(t)=\sum_{i=1}^n \gamma_i e^{-\rho_i t}$ with $0\le \rho_1<\dots<\rho_n$ and $\gamma_i>0$ for $i=1,\dots,n$, and focus on the following SVE: 
\begin{equation}\label{multifactor_CIR}
  X_t=x_0+\int_0^t G(t-s)[a-kX_s]ds +\int_0^tG(t-s) \sigma \sqrt{X_s}dW_s, 
\end{equation}
where $x_0,a\ge 0$, $k\in \R$, $\sigma>0$ and $W$ is a one-dimensional Brownian motion. Abi Jaber and El Euch have shown that there exists a unique strong nonnegative solution, see~\cite[Theorem 3.1]{AJEE}. 

We proceed as in the previous subsection and consider the associated SDE:
$$X^i_t=-\rho_i \int_0^t X^i_s ds+\int_0^t [a-kX_s]ds+\int_0^t \sigma \sqrt{X_s}dW_s,$$
with $X_t=x_0+\sum_{i=1}^n \gamma_i X^i_t$. Again, we split its infinitesimal generator~$\cL=\cL^1+\cL^2$, where $\cL^1$ and $\cL^2$ are respectively the generators associated to the SDE~\eqref{SDE_dcm_1} and the following SDE:
\begin{equation}\label{SDE_CIR_n} dX^i_t=[a-kX_t]dt+\sigma \sqrt{X_t}dW_t, \text{ with } X_t=x_0+\sum_{i=1}^n \gamma_i X^i_t. 
\end{equation}
For an initial condition~$\bfx=(x^1,\dots,x^n) \in \R^n$ such that $x=x_0+\sum_{i=1}^n \gamma_i \ge 0$, the solution of this SDE is given by $A_{\bfx}(\xi^x_t)$, with $A_{\bfx}$ defined by~\eqref{def_Axy} and $\xi^x$ is the Cox-Ingersoll-Ross process:
\begin{equation}\label{SDE_CIR}
   \xi^x_t=x+\int_0^t G(0)[a-k\xi^x_s]ds + \int_0^t G(0) \sigma \sqrt{\xi^x_s}dW_s.
  \end{equation}
The exact simulation of~$\xi^x_t$ is possible (see e.g. Alfonsi~\cite[Proposition 3.1.1]{AA_book}) but has some computational cost while the Ninomiya and Victoir scheme is only well defined for $G(0)\sigma^2\le 4a$ as pointed in Alfonsi~\cite{AA_MCOM}. Thus, it has been proposed in~\cite{AA_MCOM} the following scheme 
\begin{equation}\label{AA_scheme}\varphi(x,t,U)=\begin{cases}  e^{-\frac{ \bar{k} t}{2}}\left(\sqrt{(\bar{a}-\frac{\bar{\sigma}^2}4)\zeta_{\bar{k}}(\frac t 2 )
  +e^{-\frac{\bar{k}t}{2}} x }
+   \frac{\bar{\sigma}}{2} w(U) \right)^2 +(\bar{a}-\frac{\bar{\sigma}^2}4)\zeta_{\bar{k}}(\frac t 2) \text { if } x\ge \bar{\bf K}_2(t)\\
\mathbf{1}_{\{U \le \pi(t,x) \}}
\frac{\bar{u}_1(t,x)}{2 \pi(t,x)} + \mathbf{1}_{\{U >
  \pi(t,x) \}}\frac{\bar{u}_1(t,x)}{2(1- \pi(t,x))},\text { if } x < \bar{\bf K}_2(t)
 \end{cases}\end{equation}
 with $U$ uniformly distributed on $[0,1]$, $\bar{a}=G(0)a$, $\bar{k}=G(0)k$, $\bar{\sigma}=G(0)\sigma$, $\zeta_k(t)=\frac{1-e^{-kt}}{k}$ for $k\not=0$ and $\zeta_0(t)=t$,  $w(u)=\sqrt{3}\left(\mathbf{1}_{u> 5/6}-\mathbf{1}_{u\le 1/6}\right)$, $\pi(t,x)=\frac{1-\sqrt{ 1-\frac{\bar{u}_1(t,x)^2}{ \bar{u}_2(t,x)}}}{2}$ with
 \begin{equation}\label{2mom_CIR} 
  \bar{u}_1(t,x)  =x e^{-\bar{k}t}+\bar{a} \zeta_{\bar{k}}(t) 
  \text{ and }
\bar{u}_2(t,x)  =\bar{u}_1(t,x)^2 +\bar{\sigma}^2 \zeta_{\bar{k}}(t)[ \bar{a} \zeta_{\bar{k}}(t)/2+ xe^{-\bar{k}t}   ] , \end{equation}
and 
 $$\bar{\mathbf{K}}_2(t)=\mathbf{1}_{\{\bar{\sigma}^2 > 4\bar{a} \}}e^{\frac{\bar{k}t}{2}} \left( \left(\frac{\bar{\sigma}^2}{4}-\bar{a}\right) \zeta_{\bar{k}}\left( \frac t2\right) + \left[
  \sqrt{e^{\frac{\bar{k}t}{2}} \left(\frac{\bar{\sigma}^2}{4}-\bar{a}\right) \zeta_{\bar{k}}\left(\frac t2 \right)   } +  \frac{\bar{\sigma}}{2}  \sqrt{3t} \right]^2 \right).$$
\begin{prop}\label{prop_multifactorCIR}
Let $\bfx \in \R^n$ be such that $x=x_0+\sum_{i=1}^n \gamma_i x^i\ge 0$, and let us define $\psi_2(\bfx, t,U)=A_{\bfx}(\varphi(x,t,U))$ with $\varphi(x,t,U)$ given by~\eqref{AA_scheme}. Then, for any $f\in \Cpol{\R^n}$, there exist $C,E,\eta$ that only depend on a good sequence of~$f$ such that
\begin{align} \forall t\in (0,\eta), &\forall\bfx \in \R^n \ s.t. \ x_0+\sum_{i=1}^n \gamma_i x^i\ge 0, \notag \\
  & \left|\E \left[ f(\psi_2(\bfx, t,U))\right]-f(\bfx)-t \cL_2f(\bfx) -\frac {t^2}2\cL_2f(\bfx) \right|\le Ct^3(1+|\bfx|^E). \label{pot_2}
\end{align}
Let $\bar{\scrC}=\{\bfx \in \R^n: \forall t\ge 0, x_0+\sum_{i=1}^n \gamma_i e^{-\rho_i t}x^i \ge 0 \}$. For $\bfx \in \bar{\scrC}$, $t\ge 0$, we define 
\begin{equation}\label{scheme_multifactorCIR}
  \Psi(\bfx ,t,U):=\psi_1(\psi_2(\psi_1( \bfx,t/2),t,U),t/2).\end{equation}
Then, for any $f\in \Cpol{\R^n}$, there exist $C,E,\eta$ that only depend on a good sequence of~$f$ such that
$$  \forall t\in (0,\eta), \forall\bfx \in \bar{\scrC},\  \left|\E \left[ f(\Psi(\bfx, t,U))\right]-f(\bfx)-t \cL f(\bfx) -\frac {t^2}2\cL f(\bfx) \right|\le Ct^3(1+|\bfx|^E).$$
\end{prop}
The proof of this result is technical and relies heavily on the properties of the scheme~\eqref{AA_scheme}. It is postponed to Appendix~\ref{App_proof}.

We now turn to the simulation of the multifactor Heston model:
$$dY_t=(r-\frac{1}2 X_t)dt+ \sqrt{X_t}(\varrho dW_t+\sqrt{1-\varrho^2}dW^\bot_t),$$
where $X_t=x_0+\sum_{i=1}^n \gamma_i X^i_t$ and $(X^1,\dots,X^n)$ solves the SDE~\eqref{SDE_CIR_n}, $r\in \R$, $\varrho \in [-1,1]$ and $W^\bot$ is a Brownian motion independent of~$W$. The asset price value is given by $S_t=e^{Y_t}$, and $X_t$ is its instantaneous volatility. This model has been developed by Abi Jaber and El Euch~\cite{AJEE} and can be seen under a suitable choice of $G(t)=\sum_{i=1}^n\gamma_i e^{-\rho_i t}$ as an approximation of the rough Heston model proposed by El Euch and Rosenbaum~\cite{EER}. Let us note $\cL^H$ the infinitesimal generator of this SDE in dimension $(n+1)$. We use a similar splitting and write $\cL^H=\cL_1+\cL_2^H$, where $\cL_1$ is still the infinitesimal generator of the ODE~\eqref{SDE_dcm_1} and $\cL_2^H$ is the generator of the SDE 
$$\begin{cases}
  dY_t&=(r-\frac{1}2 X_t)dt+ \sqrt{X_t}(\varrho dW_t+\sqrt{1-\varrho^2}dW^\bot_t),\\
dX^i_t&=(a-kX_t)dt+\sigma \sqrt{X_t}dW_t, \ 1\le i\le n,\ X_t=x_0+\sum_{i=1}^n\gamma_i X^i_t. 
\end{cases}
$$
For $y>0$ and $\bfx \in \R^n$ such that $x=x_0+\sum_{i=1}^n\gamma_i x^i\ge 0$, this SDE boils down to the Heston SDE
$$\begin{cases}
  Y_t&=y+\int_0^t(r-\frac{1}2 \xi^x_t)dt+ \int_0^t \sqrt{\xi^x_s}(\varrho dW_s+\sqrt{1-\varrho^2}dW^\bot_s),\\
\xi^x_t&=x+\int_0^t G(0)(a-k\xi^x_s) ds+\int_0^t  G(0)\sigma \sqrt{\xi^x_s}dW_s, 
\end{cases}
$$
and $(X^1_t,\dots,X^n_t)=A_{\bfx}(\xi^x_t)$. Thus, by using the potential second order scheme for the Heston model given in~\cite{AA_MCOM}, we get a potential second order scheme for~$\cL_2^H$ and then for $\cL^H$. More precisely, this scheme is given by 
\begin{align}\label{scheme_L2H}
  &\bar{\psi_2}((\bfx,y),t,Z)=\Bigg(A_{\bfx}(x'), \\
  &y+\left(r-\frac{\varrho a}{\sigma} \right)t + \left(\frac{\varrho k}{\sigma} - \frac 12\right) \frac{x+x'}{2}t  +\frac{\varrho}{\bar{\sigma}}(x'-x)+\sqrt{x+Z_3 (x'-x)}\sqrt{t(1-\rho^2)}Z_2  \Bigg),\notag 
\end{align}
with $x=x_0+\sum_{i=1}^n \gamma_i x^i$, $x'=\varphi(x,t,Z^1)$ (see Eq.~\eqref{AA_scheme}), and where $Z=(Z^1,Z^2,Z^3)$ are three independent random variables with $Z^1$ uniformly distributed on $[0,1]$, $Z^2$ is a standard normal variable and $Z^3$ is a Bernoulli variable of parameter $1/2$. Setting $\bar{\psi}_1((\bfx,y),t)=(\psi_1(\bfx,t),y)$ with $\psi_1$ given by~\eqref{def_psi1}, we can then define the potential second order scheme for the multifactor Heston model
\begin{equation}\label{scheme_multifactorHeston}
  \bar{\Psi}((\bfx,y),t,Z)=\bar{\psi}_1(\bar{\psi}_2(\bar{\psi}_1((\bfx,y),t/2),t,Z),t/2).
\end{equation}
This scheme extends the scheme~\eqref{scheme_multifactorCIR}, and we could prove that it is a potential second order scheme for $\cL^H$, in the same way as for Proposition~\ref{prop_multifactorCIR}.

\subsection{Numerical results}
We present numerical results for the second order schemes presented in the previous subsection. We consider the following parameters
\begin{equation}\label{parameters}
   x_0=0.02,\ a=0.02,\ k=0.3,\ \sigma=0.3,\ \varrho=-0.7,\ S_0=1,\ r=0,
\end{equation} 
that are the same taken by Abi Jaber and El Euch~\cite{AJEE} and Richard et al.~\cite{RTY} for the rough Heston model.
We consider the time horizon $T=1$, and define the scheme by $\hat{\bfX}^N_{0}=0$, $\hat{Y}^N_0=\ln(S_0)$,  
$$(\hat{\bfX}^N_{\ell\frac TN},\hat{Y}^N_{\ell\frac TN})=\bar{\Psi}\left((\hat{\bfX}^N_{(\ell-1)\frac TN },Y^N_{(\ell-1)\frac TN }) ,\frac TN,Z_\ell \right), \ \ell\in \{1,\dots,N\},$$ 
and $\hat{X}^N_{\ell \frac TN}=x_0+ \sum_{j=1}^n \gamma_j \hat{X}^{N,j}_{\ell \frac TN }$. Let us emphasize that we have $\hat{X}^N_{\ell \frac TN}\ge 0$ for all~$\ell$, similarly as in Theorem~\ref{thm_second_order_SVE} with $\scrC=\R_+\times \R$. The function $\bar{\Psi}$ is given by~\eqref{scheme_multifactorHeston}, and $(Z_\ell)_{0\le \ell \le N}$ is an i.i.d. sequence of random variables with coordinates $Z^1_\ell\sim \mathcal{U}([0,1])$, $Z^2_\ell\sim \mathcal{N}(0,1)$, $Z^3_\ell\sim \mathcal{B}(1/2)$ that are independent.

 We first illustrate the second order convergence and the following values\footnote{These values have been obtained by using~\cite[Proposition 5.3 (1)]{AK} with $n=5$ to fit the fractional kernel with $H=0.4$.  Sharper approximation of the fractional kernel are available, but would they imply much larger values of $\rho$'s.} for the kernel:
\begin{equation*}
  \begin{tabular}{|r|r|r|r|r|}
    \hline
    $\rho_1  =0.06588906$ & $\rho_2=1.04979236$ &$\rho_3=1.78996945$ &$\rho_4=2.52111261$ &$\rho_5=3.24938890$ \\ \hline
    $\gamma_1=0.95998879$ & $\gamma_2=0.06890172$ &$\gamma_3=0.04257523$ &$\gamma_4=0.03127181$ &$\gamma_5=0.02488347$ \\ \hline
  \end{tabular}
\end{equation*} 
To visualize the \blue{asymptotic} second order convergence, it is important indeed to use coefficients with $\rho_n$ that is not too large. Heuristically, due to the exponential factor $e^{-\rho_i T/N}$, the effect of the ith coordinate starts to be non negligible when $T/N<1/\rho_i$. Since the $\rho$'s are ordered, \blue{this means that there is a transitory regime for $1/\rho_n<T/N<1/\rho_1$ and the asymptotic regime appears only when $T/N<1/\rho_n$.} Therefore, one may hope to see the \blue{asymptotic} second order convergence only when  $T/N<1/\rho_n$, which requires considerable computation time when $\rho_n$ is large \blue{since we have to run a sufficiently large number of simulations to get the statistical error smaller than the bias. However, this does not mean that the scheme is not efficient in the non-asymptotic regime $\rho_n<T/N<\rho_1$. Table~\ref{table_1} below show indeed very good approximation results with very large values of $\rho_n$.   }

Figure~\ref{Fig_convergence} plots, in function of the time step $T/N$, the empirical means of $e^{-\hat{X}^N_T/x_0}$~(Subfigure~\ref{fig:CVLaplacevol}), $(e^{\hat{Y}^N_T}-S_0)^+$ (Subfigure~\ref{fig:CVcall}) and $\frac{\hat{X}^N_T}{x_0}(e^{\hat{Y}^N_T}-S_0)^+$  (Subfigure~\ref{fig:CVSpecialcall}) calculated with $4\times 10^7$ samples. 
These functions correspond respectively to the value of the Laplace transform of $X_T$ at $1/x_0$, the value of the call option with strike $S_0$ and the value of a volatility scaled call option with the same strike.    
For each point, the corresponding 95\% confidence interval is indicated in Figure~\ref{Fig_convergence} by a red segment. \blue{From our theoretical results, we expect that the convergence to the exact value is in $O((T/N)^2)$ as $T/N \to 0$. This is indeed what we observe. In Subfigures~\ref{fig:CVLaplacevol}  and~\ref{fig:CVSpecialcall}, the plots even look like a parabola. In Subfigure~\ref{fig:CVcall}, the plot is also in line with a quadratic behaviour with respect to the time step, for $T/N \le 0.2$.  }
  
\begin{figure}[h]
  \centering
  \begin{subfigure}[h]{0.48\textwidth}
    \centering
    \includegraphics[width=\textwidth]{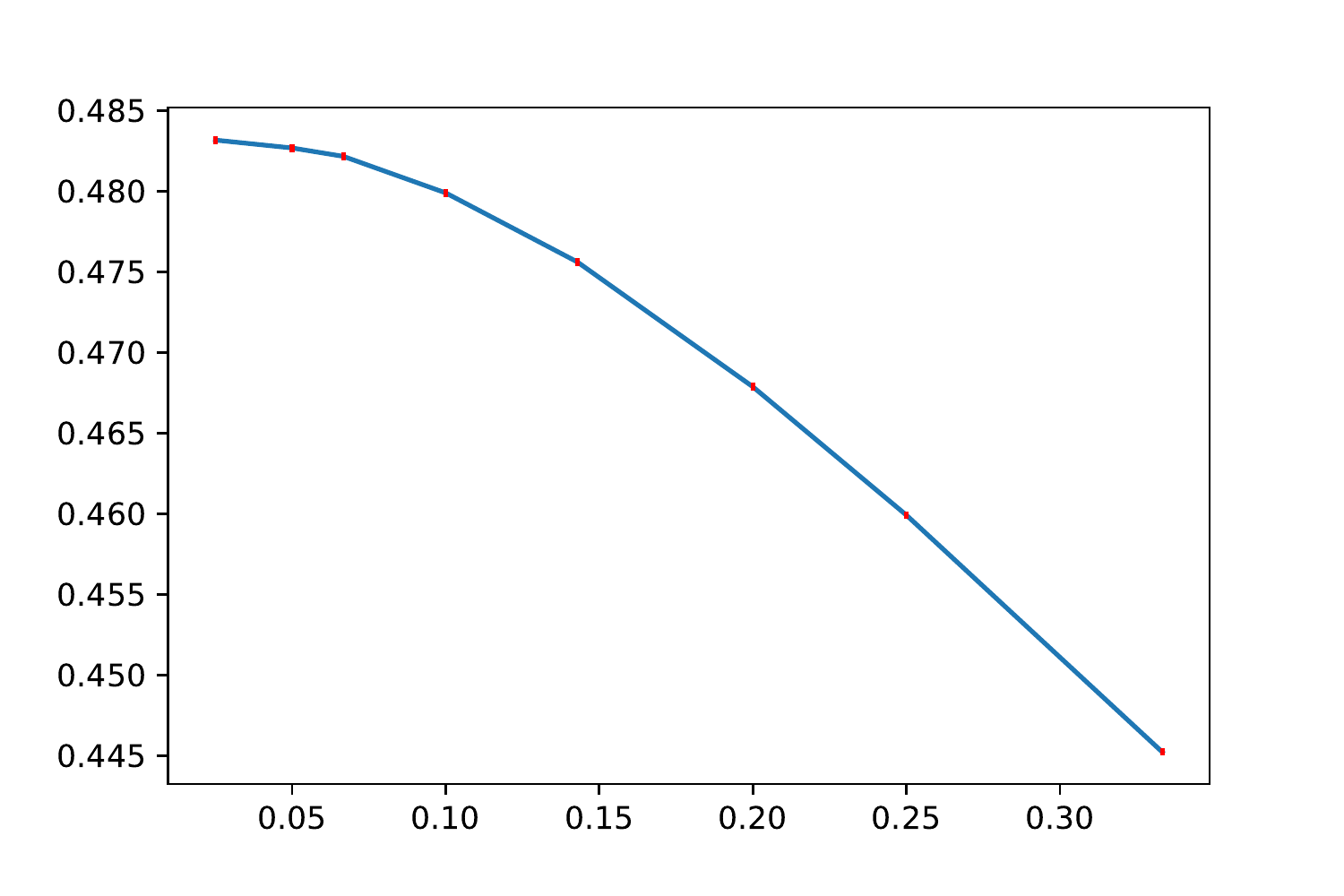}
    \caption{Empirical mean of $e^{-\hat{X}^N_T/x_0}$}
    \label{fig:CVLaplacevol}
  \end{subfigure}
  \hfill
  \begin{subfigure}[h]{0.48\textwidth}
    \centering
    \includegraphics[width=\textwidth]{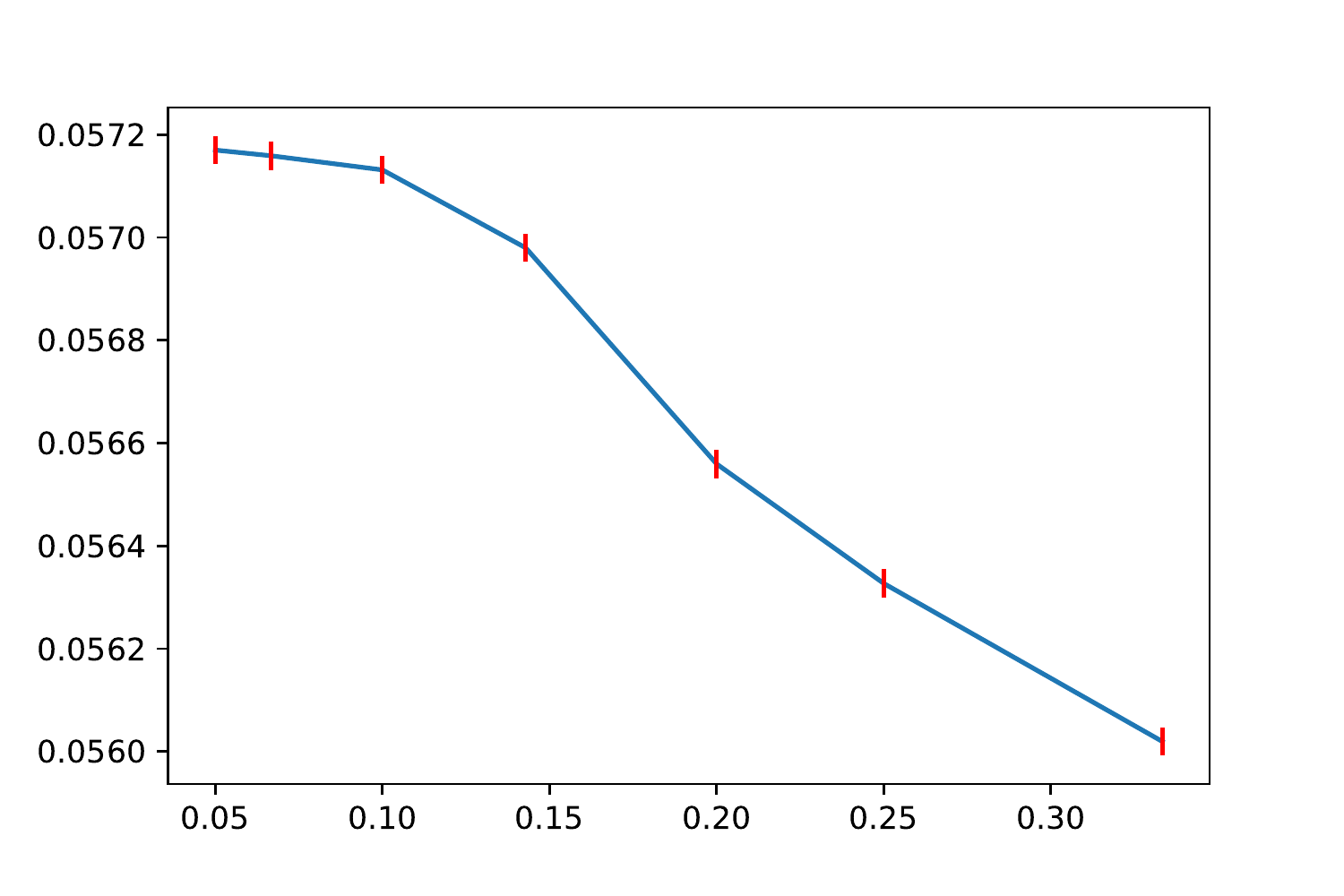}
    \caption{Empirical mean of $(e^{\hat{Y}^N_T}-S_0)^+$}
    \label{fig:CVcall}
  \end{subfigure}
  \hfill
  \begin{subfigure}[h]{0.5\textwidth}
    \centering
    \includegraphics[width=\textwidth]{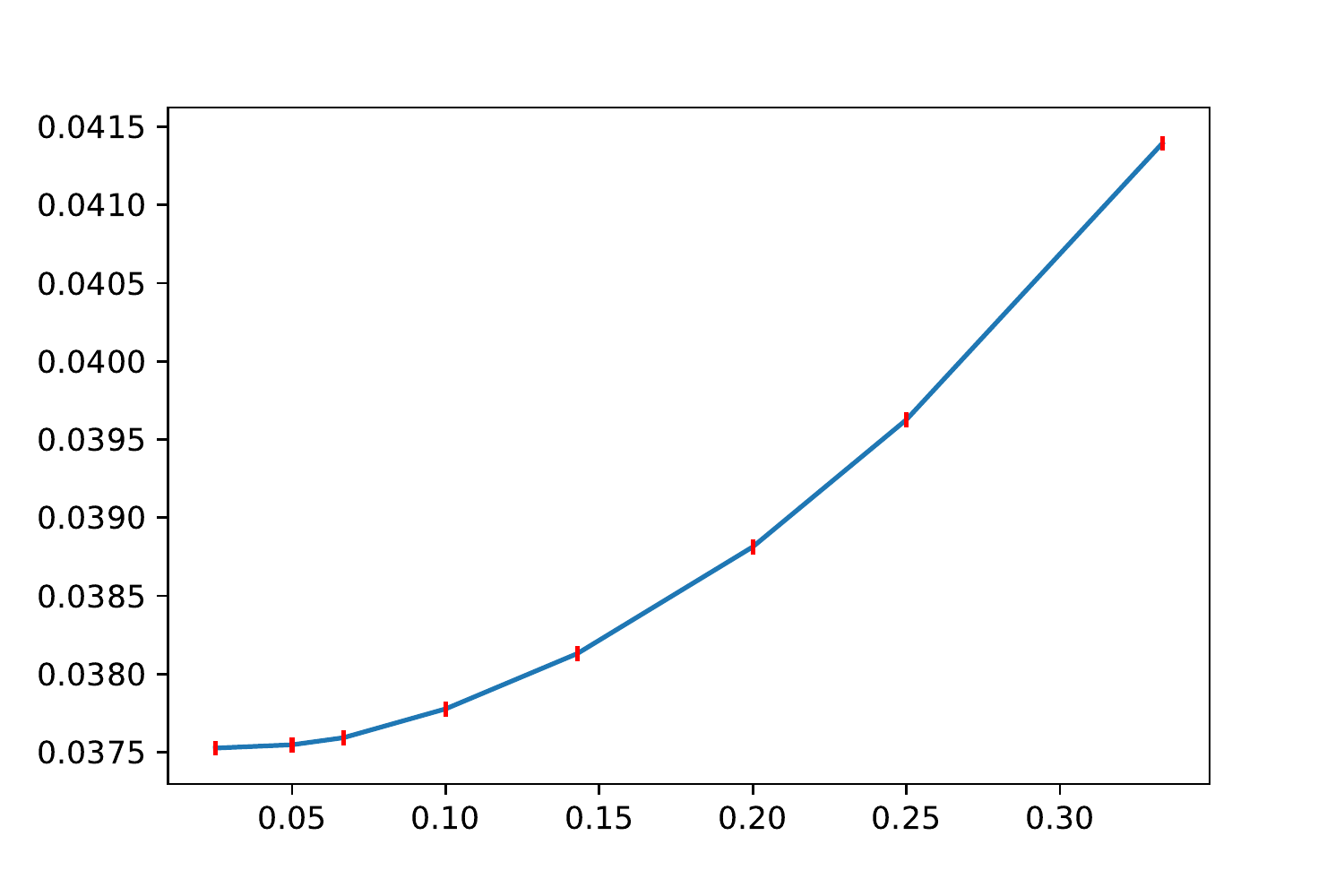}
    \caption{Empirical mean of $\frac{\hat{X}^N_T}{x_0}(e^{\hat{Y}^N_T}-S_0)^+$}
    \label{fig:CVSpecialcall}
  \end{subfigure}
  \caption{Convergence of three different empirical means in function of the time step $T/N$, with $T=1$.}\label{Fig_convergence}
\end{figure}

We now compare the convergence of our second order scheme with the Euler scheme. To do so, we consider the same numerical example as Richard et al.~\cite{RTY} and consider the pricing of a European call option in the rough Heston model with parameters~\eqref{parameters} and the fractional kernel $G_H(t)=t^{H-1/2}/\Gamma(H+1/2)$ with $H=0.1$. The Volterra Euler scheme studied by~\cite{RTY} is given by
\begin{equation}\label{Euler_Volterra}
  \begin{cases}
  Y^{E,N}_{\ell \frac{T}{N}}=Y^{E,N}_{(\ell-1) \frac{T}{N}}+(r-\frac 12 (X^{E,N}_{(\ell-1) \frac{T}{N}})_+) \frac TN +\sqrt{(X^{E,N}_{(\ell-1) \frac{T}{N}})_+} \left(\rho \Delta W_{\ell\frac T N} + \sqrt{1-\rho^2} \Delta W^\bot_{\ell\frac T N} \right)\\
  X^{E,N}_{\ell \frac{T}{N}}=x_0+\sum_{l=0}^{\ell-1} G_H\left((\ell-l)\frac{T}{N}\right)\left[ \left(a-k ( X^{E,N}_{l \frac{T}{N}})_+\right) \frac{T}{N} + \sigma \sqrt{(X^{E,N}_{l \frac{T}{N}})_+} \Delta W_{(l+1)\frac T N}\right],  
  \end{cases}
\end{equation}
for $\ell \in \{1,\dots,N\}$, with $\Delta W_{\ell\frac T N}=W_{\ell \frac{T}{N}}-W_{(\ell-1) \frac{T}{N}}$ and $\Delta W^\bot_{\ell\frac T N}=W^\bot_{\ell \frac{T}{N}}-W^\bot_{(\ell-1) \frac{T}{N}}$. One important drawback of the Volterra Euler scheme is that it requires a computation time proportional to $N^2$. However, as noticed in Alfonsi and Kebaier~\cite[Theorem 4.2]{AK}, when the kernel is a linear combination of exponential, this scheme coincides with the Euler scheme of a multidimensional SDE and can then be calculated with a computation time proportional to~$N\times n$, where $n$ is the number of exponential function involved in the combination. Here, to approximate the fractional kernel, we use the linear combination $G(t)=\sum_{i=1}^n \gamma_i e^{-\rho_i t}$ given by~\cite[Equation (6.4)]{AK} with $n=30$. However, this approximation involves very large values of $\rho$ that are useful in the minimization of $\int_0^T(G_H(t)-G(t))^2dt$, but useless for the approximation scheme. Heuristically, the coordinate corresponding to $\rho_i$ can be correctly approximated if $\rho_i T/N < 1$, and is negligible if $\rho_i T/N >> 1$ because of the exponential factor $\exp(-\rho_i T/N)$. For this reason, we keep only the $n=17$ first exponential factors.  
{\scriptsize
 \begin{table}
   \begin{tabular}{|r||l|l|l||l|l|l||l|l|l|}
     \hline
     $N$& Mean  & 95\% prec. & Time (s) & Mean   & 95\% prec. & Time (s)& Mean   & 95\% prec. & Time (s)  \\
     \hline
     10 & 0.05447 & 1.4e-4  & 5 & 0.05929   & 1.5e-4   & 5   & 0.05919   &  1.5e-4   &  3 \\
     20 & 0.05548 & 1.4e-4  & 9 & 0.05865  & 1.5e-4   & 10  & 0.05868    &  1.5e-4  & 13 \\
     40 & 0.05614 & 1.4e-4  & 18 & 0.05831  & 1.4e-4   & 21 & 0.05845  & 1.4e-4  & 50 \\
     80 & 0.05646 & 1.4e-4  & 36 & 0.05812 & 1.4e-4  & 42  & 0.05814  & 1.4e-4  & 198 \\
     160 & 0.05672 & 1.4e-4  & 72 & 0.05783 & 1.4e-4   & 79 & 0.05780  & 1.4e-4 & 745 \\
     320 & 0.05674 & 1.4e-4  & 147 & 0.05735 & 1.4e-4 & 157 & 0.05783 & 1.4e-4  & 3136\\ 
     \hline
   \end{tabular}
    \vspace{0.1cm}
   \caption{Price of the European call option in the rough Heston model by using the second order scheme (left), the multifactor Euler scheme (middle) and the Volterra Euler scheme (right). }\label{table_1}
 \end{table}
}

We have reported in Table~\ref{table_1} the results obtained with the different schemes. The approximated exact value obtained by using Fourier pricing techniques is $0.05683$. As already remarked in Alfonsi and Kebaier~\cite{AK}, there is almost no difference between the values obtained by the Euler scheme and by the Euler scheme that uses the kernel approximation. The latter one however saves significantly computation time.  Between our second order approximation scheme and the Euler scheme that both use the same multifactor approximation of the Rough kernel, there is almost no difference of computation time. The second order scheme has a much lower bias and outperforms the Euler scheme: the exact value is in the confidence interval from $N=160$, while it is still well outside for the Euler scheme with $N=320$. 

\appendix
\section{Technical proofs of Section~\ref{Sec_approx}}\label{App_proof}
\begin{proof}[Proof of Lemma~\ref{lem_NV}]
  We introduce the operators $V_k f(x)= G(0) \sum_{j=1}^k \sigma_{j,k}(x) \partial_j f(x)$ (resp. $\mathcal{V}_k f (\bfx) =\sum_{i=0}^{n-1} \sum_{j=1}^k \sigma_{j,k}(x_0+\sum_{l=1}^n\gamma_l x^l) \partial_{j+i\times d}f(\bfx)$) that are involved in the Ninomiya Victoir scheme for~\eqref{SDE_xi} (resp.~\eqref{SDE_dcm_2}), and $X_k(t,x)$ (resp. $\mathcal{X}_k(t,\bfx)$) the associated ODE. Let $\bfx =(x^1,\dots,x^n) \in (\R^d)^n$ and $x =x_0+\sum_{i=1}^d \gamma_i x^i$. We have for $i\in\{0,\dots,n-1\}$, $j\in \{1,\dots,d\}$,
$$\left(\frac{d}{dt}A_{\bfx}(X_k(t,x))\right)_{j+i\times d}=\frac 1 {G(0)}\left(\frac{d}{dt}X_k(t,x)\right)_{j}=\sigma_{j,k}(X_k(t,x))=\sigma_{j,k}(x_0+\sum_{l=1}^n A^l_{\bfx}(X_k(t,x)) ),$$ which gives
$$ \mathcal{X}_k(t,\bfx) = A_{\bfx}(X_k(t,x)), \text{ with } x =x_0+\sum_{i=1}^d \gamma_i x^i.$$
We also check easily that $\mathcal{X}_1(t_1,\mathcal{X}_2(t_2,\bfx))=A_{\bfx}(X_1(t_1,X_2(t_2,x)))$ for $t_1,t_2 \in \R$, which concludes the proof.
\end{proof}
\begin{proof}[Proof of Proposition~\ref{prop_multifactorCIR}]
When $x\ge \bar{\mathbf{K}}_2(t)$, $\varphi(x,t,U)$ is the Ninomiya-Victoir scheme for the SDE~\eqref{SDE_CIR}, where the standard normal variable has been replaced by $w(U)$ that matches the five first moments ($\E[w(U)^2]=1$, $\E[w(U)^4]=3$, $\E[w(U)^{2m+1}]=0$).  Then, by Lemma~\ref{lem_NV}, $\psi(\bfx,t,U)$ is the Ninomiya-Victoir scheme for the SDE~\eqref{SDE_CIR_n}, where the standard normal variable has been replaced by $w(U)$. It follows by~\cite[Corollary 2.3.20]{AA_book} that~\eqref{pot_2} holds.   

We now consider $x \in [0, \bar{\mathbf{K}}_2(t))$. A Taylor expansion gives for $f\in \Cpol{\R^n}$, $\bfx \in \R^n$ and $y\in \R$:
\begin{align*}
  f(A_{\bfx}(y))=& f(\bfx)+\frac{y-x}{G(0)}\sum_{i=1}^n\partial_i f(\bfx)+ \frac 12 \left(\frac{y-x}{G(0)}\right)^2\sum_{i_1,i_2=1}^n\partial_{i_1}\partial_{i_2} f(\bfx) \\&+\left(\frac{y-x}{G(0)}\right)^3\int_0^1\frac{(1-\alpha)^2}2 \sum_{i_1,i_2,i_3=1}^n\partial_{i_1}\partial_{i_2}\partial_{i_3}f\left(\bfx +\alpha\frac{y-x}{G(0)}\right)  d\alpha
\end{align*}
Let $C>0,E\ge 1$ be such that $\sum_{i_1,i_2,i_3=1}^n|\partial_{i_1}\partial_{i_2}\partial_{i_3}f(\bfx)|\le C(1+|\bfx|^E)$. 
Since $\E[\varphi(x,t,U)]=\bar{u}_1(t,x)=\E[\xi^x_t]$ and $\E[\varphi(x,t,U)^2]=\bar{u}_2(t,x)=\E[(\xi^x_t)^2]$, we get 
\begin{align*}
  |\E[f(A_{\bfx}(\varphi( x,t,U)))]-\E[f(A_{\bfx}(\xi^x_t))]| \le &\frac{C}{6}(1+2^{E-1}|\bfx|)\left(\E\left[\frac{|\varphi( x,t,U)-x|^3}{G(0)^3}\right]+\E\left[\frac{|\xi^x_t-x|^3}{G(0)^3}\right] \right)\\
  &+\frac{C}{6}2^{E-1}\left( \E\left[\frac{|\varphi( x,t,U)-x|^{3+E}}{G(0)^{3+E}}\right]+\E\left[\frac{|\xi^x_t-x|^{3+E}}{G(0)^{3+E}}\right]\right).
\end{align*}
Since $\bar{\mathbf{K}}_2(t)=O(t)$, we get that $\E[|\xi^x_t-x|^q]+ \E[|\varphi(x,t,U)-x|^q]\le C_q t^q$ (see the proof of~\cite[Proposition 3.3.5]{AA_book} for details), and therefore~\eqref{pot_2} holds by using~\cite[Remark 2.3.7]{AA_book}.  

Then,  the second part of the claim is just then an application of the splitting technique~\cite[Corollary 2.3.14]{AA_book}.
\end{proof}

\bibliographystyle{abbrv}
\bibliography{biblio_CMP}

\end{document}